\documentclass[11pt,leqno,oneside]{amsart}

\usepackage{graphicx}
\usepackage{indentfirst,csquotes}
\usepackage{multicol}
\usepackage{amssymb,amsthm,amsmath,amsfonts}
\usepackage{xcolor,paralist,hyperref,fancyhdr,etoolbox}
\usepackage{comment}

\hypersetup{colorlinks=true, linkcolor=black, filecolor=black, urlcolor=black}

\newcommand{\AuthorBlock}[3]{%
  \noindent\textbf{#1}\\%
  #2\\%
  E-mail: \texttt{#3}%
  \par\vspace{1em}%
}
\makeatletter
\patchcmd{\@settitle}
  {\uppercasenonmath\@title}
  {}
  {}
  {}

\patchcmd{\@setauthors}
  {\MakeUppercase{\authors}}
  {\authors}
  {}
  {}
\makeatother
\newtheorem{theorem}{Theorem}[section]
\newtheorem{definition}[theorem]{Definition}

\newtheorem{lemma}[theorem]{Lemma}
\newtheorem{proposition}[theorem]{Proposition}
\newtheorem{corollary}[theorem]{Corollary}
\newtheorem{conjecture}[theorem]{Conjecture}
\newtheorem*{remark}{Remark}
\newtheorem{question}[theorem]{Problem}

\newcommand{\D}{\mathbb{D}}

\newcommand{\C}{\mathbb{C}}

\newcommand{\E}{\mathbb{E}}
\newcommand{\RCP}{\text{RCP}}

\newcommand{\pac}{{\mathcal{P}}}

\newcommand{\eps}{\varepsilon}

\title[A Gauss-Bonnet-Type Dichotomy]
{Geometry of Unimodular Random Hyperbolic Polyhedra}

\author[Ge, Lu, Wang and Zhou]
{Huabin Ge, Yangxiang Lu, Chuwen Wang, and Tian Zhou}

\date{\today}

\begin{document}

\begin{abstract}
We develop a geometric and probabilistic theory of unimodular random infinite trivalent hyperbolic polyhedra in \(\mathbb H^3\), allowing ordinary, ideal, and hyperideal vertices. Our results relate local dihedral geometry to global conformal type and random-walk behavior. A central role is played by a geometric characteristic \(L_f(P)\), determined by the dihedral angles incident to a face \(f\). We establish the unimodular Gauss–Bonnet identity
\[
\mathbb E[L_f(P)]=2\pi-\frac{\pi}{3}\mathbb E[\deg(f)].
\]
Together with circle-pattern uniformization, this yields a sharp parabolic–hyperbolic dichotomy governed by the expected geometric characteristic. We also introduce polyhedral soficity and prove that admissible Benjamini–Schramm limits of finite trivalent hyperbolic polyhedra are parabolic.

We further establish a quantitative Ring Lemma for regular circle patterns, providing degree-sensitive geometric control without assuming bounded degrees. Combined with unimodular probability methods, this estimate enables us to identify the Poisson boundary of the face random walk and establish positive linear hyperbolic speed. Our results extend the geometric-characteristic theory beyond ideal polyhedra, connecting three-dimensional geometry, unimodular Gauss–Bonnet theory, and stochastic boundary behavior.
\end{abstract}

\maketitle
\section{Introduction}
Since the foundational work of Andreev, Thurston, Rivin, and Rivin–Hodgson, the realization and rigidity theory of hyperbolic polyhedra has developed into a rich subject linking discrete geometry, circle patterns, and geometric analysis \cite{Andreev1970,Andreev1970finite,RivinHodgson1993,Rivin1994,Rivin1996,RoederHubbardDunbar2007,Thurston1979}. Subsequent developments have extended this theory to hyperideal polyhedra and prescribed-angle circle patterns \cite{BaoBonahon2002,BobenkoSpringborn2004,Schlenker2005}. For infinite hyperbolic polyhedra, however, a fundamentally different set of questions arises. Instead of asking only whether prescribed local data admit a geometric realization, one may ask how the local geometry determines the behavior of the polyhedron at infinity. When the combinatorial structure and dihedral-angle data are random, this becomes a problem at the intersection of hyperbolic geometry, discrete conformal geometry, and probability.

The purpose of the present paper is to develop this viewpoint for unimodular random infinite trivalent hyperbolic polyhedra in $\mathbb{H}^{3}$, including polyhedra with ordinary, ideal, and hyperideal vertices. Our guiding question is:

\textit{\textbf{Can the global conformal and probabilistic geometry of an infinite random hyperbolic polyhedron be characterized through geometric quantities determined locally by its three-dimensional dihedral-angle data?}}

A fundamental motivation comes from the theory of geometric characteristics of circle patterns, initiated in the work of Ge-Lin \cite{GeLin2024} and developed in subsequent studies of discrete curvature, circle-pattern realization, and hyperbolic polyhedra. A geometric characteristic is a quantity computable from prescribed local geometric data, designed to capture information about the existence or global nature of a geometric realization. Its significance lies in the possibility that a relatively simple local observable may encode a much more complicated global geometric phenomenon. The transition from finite to infinite geometry makes this perspective particularly compelling. In the finite setting, geometric characteristics are closely related to discrete curvature and Gauss-Bonnet identities. In the infinite setting, ordinary global curvature sums need not converge, and a deterministic local characteristic generally does not determine the conformal type of an arbitrary infinite configuration. Unimodularity provides a natural replacement for finite averaging through the mass transport principle \cite{AldousLyons2007}. It allows one to formulate curvature identities at the level of expectations and thereby connects local geometry with statistical properties of infinite structures.

The probabilistic side of this program is informed by the theory of unimodular random planar maps. In particular, Angel-Hutchcroft-Nachmias-Ray established fundamental parabolic–hyperbolic dichotomies for unimodular random triangulations, together with their connections to invariant amenability, circle packing, random walks, and Poisson boundaries \cite{angel2016unimodular,angel2018hyperbolic}. These results provide powerful tools for studying random infinite planar structures. The geometric questions considered here, however, require additional structures not present in ordinary tangency circle packings. Prescribed intersection angles must be realized by regular circle patterns, the corresponding configurations must produce genuine hyperbolic polyhedra, and suitable geometric quantities must be extracted directly from the three-dimensional dihedral data.

The necessary realization and uniformization theory has been developed in a series of works on infinite circle patterns and hyperbolic polyhedra \cite{IIP2026}-\cite{GeRCP}. In particular, our previous paper \cite{IIP2026} established a probabilistic theory for random infinite ideal angled graphs and ideal hyperbolic polyhedra. The present work advances this program beyond the ideal setting to general trivalent hyperbolic polyhedra. This passage is not merely formal. Ideal vertices satisfy special linear angle relations, whereas ordinary and hyperideal vertices require a nonlinear geometric construction. Moreover, the corresponding regular circle patterns present additional quantitative difficulties, especially in random triangulations with unbounded degrees. Resolving these difficulties is central to the present paper. Our results develop three closely related aspects of the theory: a geometric-characteristic formulation of the conformal-type dichotomy, a restriction on finite polyhedral approximation, and quantitative circle-pattern estimates leading to random-walk boundary theorems.

\subsection{Infinite trivalent hyperbolic polyhedra and geometric characteristics}
A convex hyperbolic polyhedron is an intersection
\[
P=\bigcap_{i\in I}H_{i}\subset\mathbb{H}^{3}
\]
of closed hyperbolic half-spaces. We consider infinite trivalent hyperbolic polyhedra (THPs), in which every vertex is incident to exactly three faces. Ordinary vertices in $\mathbb{H}^3$, ideal vertices on $\partial_\infty\mathbb{H}^3$, and hyperideal vertices are all allowed. The trivalent assumption is particularly natural from the viewpoint of duality: the dual combinatorial structure is a disk triangulation. Consequently, the local geometry at each polyhedral vertex is encoded by three dihedral angles, and these angles determine an auxiliary Euclidean triangle. This leads to the central geometric quantity of the paper.

\begin{definition}[Geometric characteristic number of a THP]
Let $P$ be a trivalent hyperbolic polyhedron, and let $v$ be a vertex incident to the three faces $f_{1}, f_{2}, f_{3}$. Write $e_{ij}=f_{i}\cap f_{j}$, and denote the dihedral angle along $e_{ij}$ by $\Theta_{ij}$. For $\{i,j,k\}=\{1,2,3\}$, define
\[
\theta_{v}^{f_{i}}=\arccos\left(\frac{1+\cos\Theta_{ij}+\cos\Theta_{ik}-\cos\Theta_{jk}}{2\sqrt{1+\cos\Theta_{ij}}\sqrt{1+\cos\Theta_{ik}}}\right)
\]
Equivalently, $\theta_{v}^{f_{i}}$ is the corresponding angle of the auxiliary Euclidean triangle determined by the three side lengths
\[
\ell_{ij}=\sqrt{2+2\cos\Theta_{ij}}.
\]
For a face $f$ of $P$,  define its geometric characteristic number by
\begin{equation}
L_{f}(P)=2\pi-\sum_{v\in f}\theta_{v}^{f}.
\end{equation}
The collection $L(P)=(L_{f}(P): f\in F(P))$ is called the geometric characteristic of $P$.
\end{definition}

The importance of this definition lies in its geometric origin. The quantities $\theta_v^f$ are not freely prescribed combinatorial weights. They are canonically recovered from the dihedral-angle geometry of the polyhedron. In particular, the characteristic is independent of a choice of circle-pattern radii and depends only on the corresponding marked polyhedral structure. This construction also connects naturally with the earlier curvature theory of circle patterns. In the finite setting, the corresponding curvature quantity arises from the variational geometry of circle-pattern radii; the characteristic above represents the curvature at the unit-radius configuration \cite{BobenkoSpringborn2004,GeLin2024,Springborn2003}. The local angular construction therefore retains the geometric meaning underlying the original characteristic theory.

The distinction between ideal and non-ideal vertices is essential. At a trivalent ideal vertex, the three dihedral angles satisfy a linear identity, and the characteristic admits a simple expression in terms of the dihedral angles along the boundary of a face. For general ordinary or hyperideal vertices, that linear identity is no longer available. The characteristic instead depends nonlinearly on the full local angle configuration. Nevertheless, the two theories are compatible. Proposition \ref{prop:compatibility-ideal} shows that on the ideal locus the present $L_f(P)$ agrees exactly with the characteristic $T_P(f)$ used in \cite{IIP2026}. Thus the geometric characteristic provides a unified language for trivalent ideal and non-ideal hyperbolic polyhedra. One of the main conclusions of the paper is that this locally defined geometric quantity continues to characterize the global conformal type after unimodular averaging.

\subsection{Angled disk triangulations and geometric realization}

The dual viewpoint is essential for introducing unimodularity. The dual 1-skeleton $M(P)$ of a trivalent hyperbolic polyhedron is a disk triangulation. Each edge of $M(P)$ inherits the dihedral angle of the corresponding polyhedral edge. We work with the following class of marked triangulations.

\begin{definition}[Angled disk triangulation]
An angled disk triangulation (ADT) is a pair $(G,\Theta)$, where $G=(V, E)$ is a disk triangulation graph and $\Theta: E \to [0, \pi)$
is an angle assignment. A rooted ADT is a triple $(G, \rho, \Theta)$, where $\rho\in V$ is the distinguished root.
\end{definition}

A random rooted ADT is called unimodular if its law satisfies the mass transport principle. Under the THP/ADT duality, a root vertex corresponds to a distinguished face of the hyperbolic polyhedron. The use of unimodularity is natural in infinite geometry. In a finite graph, one can choose a uniformly random root and average local observables over all vertices. An infinite graph has no uniform probability distribution on its vertex set. The mass transport principle provides the distributional analogue of this finite averaging procedure and makes it possible to formulate meaningful infinite-volume curvature identities.

To connect a marked triangulation with an actual hyperbolic polyhedron, we use regular circle patterns (RCPs) with prescribed intersection angles. The theory developed in \cite{GeRCP} gives existence, rigidity, and uniformization results under the appropriate angle conditions (Z1)-(Z4). An RCP realization determines a family of circles at infinity, which in turn determines hyperbolic half-spaces. Under the RCP-THP correspondence, their intersection yields a trivalent hyperbolic polyhedron whose dihedral angles agree with the prescribed intersection-angle data. This geometric realization step is indispensable. It transforms a random marked triangulation into a random three-dimensional polyhedral object and makes it possible to interpret local curvature and global conformal type in genuine hyperbolic-geometric terms. Throughout the paper, the tameness assumption requires the intersection angles to be uniformly separated from $\pi$. The precise admissibility and strengthened cycle-separation conditions are recalled in Section \ref{sec:prelim}.

\subsection{A unimodular Gauss-Bonnet principle and conformal-type dichotomy}
Our first main result establishes an infinite-volume Gauss–Bonnet identity for the geometric characteristic. Let $(G,\rho,\Theta)$ be a unimodular random ADT, and let $L_\rho(G,\Theta)$ denote the geometric characteristic number of its root vertex, corresponding to the characteristic of the root face of the dual hyperbolic polyhedron. Under the integrability and admissibility assumptions, we prove
\begin{equation}
\mathbb{E}[L_{\rho}(G,\Theta)]=2\pi-\frac{\pi}{3}\mathbb{E}[\deg_{G}(\rho)].
\end{equation}

The remarkable feature of this formula is the cancellation of angular dependence. Pointwise, $L_\rho(G,\Theta)$ is determined by the local dihedral-angle configuration and generally depends on these angles nonlinearly. Its expectation, however, depends only on the expected degree. The underlying identity is universal. At every triangular face of the dual ADT, the three auxiliary Euclidean angles sum to $\pi$. The mass transport principle converts this local angle identity into the expected-curvature formula above. In particular, the nonlinearity of the geometric characteristic is compatible with a universal Gauss-Bonnet law. This is the precise sense in which the local three-dimensional geometric information and the global combinatorial averaging interact.

The identity also has a finite counterpart. For a finite trivalent convex hyperbolic polyhedron,
\[
\sum_{f\in F(P)}L_f(P)=4\pi.
\]
Thus the characteristic extends the familiar relationship between total curvature and spherical topology to an infinite-volume probabilistic setting. The essential geometric consequence is that the sign of the expected characteristic determines the conformal type of the associated infinite polyhedron.

\begin{theorem}[ADT dichotomy]
\label{thm:dichotomy}
Let $(G,\rho,\Theta)$ be an infinite, locally finite, simple, one-ended, ergodic, tame unimodular random ADT satisfying (Z1)-(Z4). If $\mathbb{E}[\deg(\rho)]<\infty$, then
\[
\mathbb{E}[L_\rho(G,\Theta)]=0
\quad\Longleftrightarrow\quad
G\ \text{is almost surely VEL-parabolic},
\]
whereas
\[
\mathbb{E}[L_\rho(G,\Theta)]<0
\quad\Longleftrightarrow\quad
G\ \text{is almost surely VEL-hyperbolic}.
\]
Moreover, the two regimes admit the corresponding characterizations in terms of RCP type and invariant amenability, together with the random-walk conclusions stated in Theorem \ref{thm:dichotomya}.
\end{theorem}

Through the geometric correspondence between ADTs and THPs, we obtain a classification directly in terms of polyhedral geometry.

\begin{theorem}[THP dichotomy]
\label{thm:polyhedral-dichotomy}
Let $P$ be a tame, infinite, ergodic, unimodular random trivalent hyperbolic polyhedron satisfying (Z3)-(Z4). Let $f$ be its root face. If $\mathbb{E}[\deg(f)]<\infty$, then
\[
\mathbb{E}[L_{f}(P)]=0 \iff P \text{ is a.s. parabolic,}
\]
and
\[
\mathbb{E}[L_{f}(P)]<0 \iff P \text{ is a.s. hyperbolic.}
\]
\end{theorem}

These results may be viewed as an intersection of two previously developed theories. On the geometric side, the characteristic arises from the curvature theory of regular circle patterns and is linked to the realization and uniformization of infinite hyperbolic polyhedra. On the probabilistic side, unimodularity and the classification of random planar triangulations provide the mechanism by which an expected local curvature determines a global conformal type.

The resulting criterion is notable for its simplicity: despite the nonlinear dependence of the characteristic on the dihedral angles, the parabolic-hyperbolic distinction is governed by the sign of a single expected geometric quantity. This is the local-to-global principle underlying our theory. Its geometric content is not the introduction of another arbitrary curvature on a planar graph, but the identification of a natural three-dimensional angle characteristic whose expectation exactly detects the conformal type of the associated polyhedron.

\subsection{Polyhedral approximation and soficity}

A second direction concerns the relationship between finite and infinite hyperbolic polyhedra. For unimodular random graphs, soficity asks whether a given random rooted infinite graph can be obtained as a Benjamini-Schramm limit of finite graphs \cite{AldousLyons2007}. In the present setting, the geometric structure suggests a more restrictive approximation problem. 

We call a unimodular random infinite trivalent hyperbolic polyhedron polyhedrally sofic if its rooted marked dual triangulation is a local weak limit of uniformly face-rooted finite trivalent hyperbolic polyhedra, including their dihedral-angle data. This definition reflects the distinction between approximating a graph and approximating a geometric polyhedron. In the latter problem, the finite objects must possess compatible three-dimensional hyperbolic realizations. We show that such approximation is subject to a strong restriction.

\textbf{Finite-approximation obstruction (Theorem \ref{thm:polyhedral-bs-parabolicity})}. \textit{Let \((P_n)\) be a sequence of finite trivalent hyperbolic polyhedra with \(|F(P_n)|\to\infty\), and suppose that their uniformly face-rooted marked dual triangulations converge locally in distribution. If the limiting marked triangulation is almost surely infinite, locally finite, simple, one-ended, tame, and satisfies \((Z1)\)–\((Z4)\), then every ergodic component of its law is RCP-parabolic. Consequently, the associated infinite trivalent hyperbolic polyhedron is almost surely parabolic.}

Remarkably, no moment or uniform-integrability assumption on the degrees is required. In particular, an admissible hyperbolic unimodular THP cannot be polyhedrally sofic (Corollary \ref{cor:hyperbolic-not-polyhedrally-sofic}). The mechanism is closely related to the finite and infinite Gauss–Bonnet identities. The spherical topology of finite polyhedral boundaries forces the average degree of their dual triangulations to approach six. 
Local weak convergence preserves the resulting upper bound on the limiting expected degree. The unimodular dichotomy supplies the reverse inequality and forces equality.

Consequently, $\mathbb E[L_\rho]=0$. This conclusion illustrates how a finite topological constraint can persist in an infinite random geometric limit. It also distinguishes finite approximation by spherical polyhedra from unrestricted sofic approximation, in which finite graphs or triangulations of other surface topologies may be permitted. The question of whether the corresponding parabolicity condition is also sufficient for polyhedral soficity remains open.

\subsection{A quantitative Ring Lemma and probabilistic boundary theory}

The Gauss-Bonnet principle determines the conformal type, but it does not by itself provide sufficient quantitative control over the geometry of random walks. In the hyperbolic regime, a more delicate problem is to understand the asymptotic behavior of the face random walk through the RCP realization. Here the main obstacle is the absence of a uniform bound on vertex degrees. Classical Ring-Lemma estimates compare the radii of neighboring circles \cite{HeSchramm1995,RodinSullivan1987}, but uniform bounds of the usual kind are generally unavailable in random triangulations with unbounded degrees.

The difficulty is particularly significant for prescribed-angle regular circle patterns. Unlike ordinary circle packings, neighboring circles need not be tangent; unlike the ideal configurations studied in \cite{IIP2026}, the local geometry is not constrained by the ideal-vertex identities. The proof therefore requires quantitative control of intersecting-circle configurations in the general non-ideal setting. We establish the following degree-sensitive estimate.

\begin{lemma}[Refined Ring Lemma for RCP]
\label{lem:intro-refined-ring}
Let $(G,\Theta)$ be a tame ADT satisfying (Z1)-(Z4), with (Z2) strengthened to the uniform condition (Z2'). Let $r:V(G)\to(0,\infty)$ be the radius function of an RCP realization. Then there exists a constant $C=C(\varepsilon_{0},\varepsilon_{1})>0$ depending only on the uniform cycle-separation and angle parameters such that, for every edge $u\sim v$,
\[
\frac{r(v)}{r(u)}>e^{-CS(u)}
\]
where $S(u)=\sum\limits_{w\sim u}\deg_{G}(w)$ is the flower degree of $u$.
\end{lemma}
The feature of this estimate is its explicit dependence on local combinatorial complexity rather than on a uniform degree bound. The proof requires additional geometric estimates for collections of intersecting circles, together with a propagation argument around a circle flower. These estimates replace the special local geometry available in the ideal setting and make it possible to control radius ratios under general prescribed intersection angles.

The bound is particularly suited to unimodular random triangulations. Indeed, the mass transport principle gives
\[
\mathbb E[S(\rho)]=\mathbb E[\deg_G(\rho)^2]
\]
and
\[
\mathbb E[\deg_G(\rho)S(\rho)]
\leq
\mathbb E[\deg_G(\rho)^3].
\]
Thus, under a finite third-moment assumption, the degree-sensitive Ring Lemma provides the integrability needed to control logarithmic changes of circle radii along the random walk.

We obtain the following boundary identification.

\begin{corollary}[Boundary identification]
\label{thm:0.2}
Assume the unimodular ADT $(G,\rho,\Theta)$ is simple, one-ended, tame, and RCP-hyperbolic, and satisfies  (Z1), (Z3), (Z4), and (Z2$'$). Suppose further that $\mathbb{E}[\deg_{G}(\rho)^{3}]<\infty$. Let $(X_{n})_{n\ge0}$ be simple random walk on $G$ started at $\rho$. Then, conditional on $(G,\rho,\Theta)$, almost surely, the Euclidean and hyperbolic centers of the corresponding circles converge to a common random limit $\Xi\in\partial\mathbb{D}$. The distribution of $\Xi$ has full support and no atoms, and the circle $\partial\mathbb{D}$, equipped with the corresponding exit measures, realizes the Poisson boundary of $G$.
\end{corollary}

The corresponding face random walk on the dual trivalent hyperbolic polyhedron has the same boundary description. We also establish positive linear escape in the hyperbolic metric.

\begin{corollary}[Positive speed]
\label{thm:positive-speed}
Under the same hypotheses as in Corollary \ref{thm:0.2},
\[
\lim_{n\to\infty}\frac{d_{h}(z_{h}(\rho), z_{h}(X_{n}))}{n} = \lim_{n\to\infty}\frac{-\log r(X_{n})}{n} = \lambda > 0 
\]
almost surely. If the marked random rooted ADT is ergodic, then $\lambda$ is almost surely constant.
\end{corollary}

Together with the ideal-polyhedral theory developed in \cite{IIP2026}, our results contribute to a broader geometric and probabilistic theory of infinite hyperbolic polyhedra. The refined Ring Lemma extends the boundary methods of \cite{angel2016unimodular,angel2018hyperbolic,IIP2026} to the non-ideal setting. Further directions include intrinsic polyhedral boundaries, infinite-volume variational structures, and extensions beyond trivalent polyhedra. More specifically, combining our finite-approximation methods and ideal-polyhedral theory with the tree-based local-limit construction of Budd-Curien  \cite{Budd-Curien} may provide an approach to their conjecture that the limiting random punctured surface is conformally parabolic after filling its punctures.

\subsection{Organization}
Section 2 recalls the geometric and probabilistic preliminaries, including angled disk triangulations, regular circle patterns, the RCP-THP correspondence, geometric characteristics, unimodularity, and the relevant angle conditions. Section 3 establishes the unimodular Gauss–Bonnet formula and the parabolic–hyperbolic dichotomies for ADTs and THPs. Section 4 introduces polyhedral soficity and proves the parabolic nature of admissible local limits of finite trivalent hyperbolic polyhedra. Section 5 proves the quantitative Ring Lemma for regular circle patterns with prescribed intersection angles. Section 6 applies the refined geometric estimates to the face random walk, establishing convergence to the circle-pattern boundary, identifying the Poisson boundary, and proving positive hyperbolic speed. Finally, Section 7 discusses further questions concerning geometric boundaries, variational interpretations of the characteristic, and extensions beyond trivalent polyhedra.

\subsection*{Acknowledgements.}
The authors thank Timothy Budd, Nicolas Curien, Puchun Zhou, and Longsong Jia for helpful discussions. The first three authors thank Xin Sun and Jian Ding for their valuable comments and suggestions. The fourth author thanks Gang Tian for his support and encouragement. The first author acknowledges support from the National Natural Science Foundation of China (Grant Nos. 12341102, 12122119, and 12525103).

\section{Preliminaries}
\label{sec:prelim}

In this section we collect the geometric, combinatorial, and probabilistic
ingredients required for this paper. 

\subsection{Angled disk triangulations and regular circle patterns} 
We recall the definitions of ADTs and tame ADTs introduced above.

\begin{definition}[ADT]
    A \emph{disk triangulation} is a triple $(V, E, F)$ that is a locally finite triangulation of the unit disk $\D$. A \emph{disk triangulation graph} is a pair $(V, E)$ that can be realized by some disk triangulation. An \emph{angled disk triangulation graph} is a pair $(G, \Theta)$ where $G=(V,E)$ is a disk triangulation graph and $\Theta:E\to[0,\pi)$ assigns an angle to each edge. 
\end{definition}

An ADT $(G,\Theta)$ is \emph{tame} if there exists $\eps>0$ such that
\[
\Theta(e)\in[0,\pi-\eps]\qquad\text{for all }e\in E.
\]
For a unimodular random rooted ADT, we say it is tame if there exists a deterministic $\eps>0$ that works almost surely (equivalently, $\eps$ can be chosen uniformly on the support of the law).

\begin{definition}[Angle conditions (Z1)--(Z4)\cite{GeRCP}]
Let $(G,\Theta)$ be an ADT. When needed, the following conditions are imposed:
\begin{enumerate}[(Z1)]
\item If $e_1,e_2,e_3$ are the three edges forming the boundary of a triangular face and
\[\Theta(e_1)+\Theta(e_2)+\Theta(e_3)>\pi,\]
then for every permutation $\{i,j,k\}=\{1,2,3\}$ we have
\[\Theta(e_i)+\Theta(e_j)<\pi+\Theta(e_k).\]

\item If $e_1,e_2,\dots,e_s$ form a simple closed curve in $G$ that is not the boundary of any triangular face, then
\[\sum_{i=1}^s \Theta(e_i) < (s-2)\pi.\]

\item If $e_1$ and $e_2$ are \emph{homologically non-adjacent} edges in $G$, then
\[\Theta(e_1)+\Theta(e_2)<\pi.\]
Here, an arc $\Gamma$ formed by two adjacent edges in $E$ is called homologically adjacent if there is an edge in $E$ that connects the two endpoints of $\Gamma$. Conversely, we call it homologically non-adjacent.

\item If $e_1,e_2,e_3$ are the three edges forming the boundary of a triangular face, then for every permutation $\{i,j,k\}=\{1,2,3\}$ we have
\[\cos\Theta(e_i)+\cos\Theta(e_j)\cos\Theta(e_k)\ge 0.\]
\end{enumerate}
\end{definition}

The strengthened condition $(Z2')$ is the following uniform version of
$(Z2)$: there exists $\eps_0=\eps_0(G,\Theta)>0$ such that, for every
simple closed curve $e_1,\dots,e_s$ in $G$ that is not the boundary of a
triangular face,
\[
\sum_{i=1}^s \Theta(e_i) \le (s-2)\pi-\eps_0.
\]

Circle patterns provide a discrete conformal realization of ADTs.
The Koebe-Andreev-Thurston circle packing theorem and its refinements
show that many planar triangulations can be realized by configurations of
circles.  For ADTs, regular circle patterns (RCP) with
prescribed intersection angles provide the corresponding realization
theory, together with existence and rigidity results
\cite{BeardonStephenson1990,BeardonStephenson1991,GeRCP,He1996,He1999,HeSchramm1993,HeSchramm1995,
Koebe1936,RodinSullivan1987,Thurston1979}.
Variational approaches also give powerful methods for studying circle
patterns with prescribed intersection angles
\cite{BobenkoSpringborn2004,ColinDeVerdiere1991,
Springborn2003,Stephenson2005}.

\begin{definition}[RCP]
A circle pattern
\(
\mathcal{P}=\{C_v\}_{v\in V}
\)
is called an RCP if
its carrier graph
\(
G(\mathcal{P})=(V,E)
\)
is a disk triangulation graph.
\end{definition}

The following theorems of Ge–Jia–Yu–Zhou establish the existence and rigidity of regular circle-pattern realizations:

\begin{theorem}[Existence \cite{GeRCP}]
\label{infinite_existence}
Let \((G,\Theta)\) be an ADT satisfying
\[
    \sup_{e\in E}\Theta(e)<\pi.
\]
If conditions \emph{(Z1)--(Z3)} hold, then there exists an RCP
realizing \((G,\Theta)\). If conditions \emph{(Z2)} and \emph{(Z4)}
hold, then there exists an embedded circle pattern \(P\)
weakly realizing \((G,\Theta)\); that is, \(G\) is a subgraph of the
contact graph \(G(P)\).
\end{theorem}

\begin{theorem}[Rigidity \cite{GeRCP}]\label{thm-intro-rigidity-cp}
	Let $(G,\Theta)$ be an ADT that satisfies ($Z_2$) and ($Z_4$). Let $\mathcal{P}$ and $\mathcal{P}'$ be two RCPs that realize $(G,\Theta)$. 
\begin{itemize}
\item [(1)] If $\pac$ and $\pac'$ are locally finite in $\mathbb{D}$, then there is a M\"obius transformation $h$ such that $\mathcal{P}'=h(\mathcal{P})$.\\[-10pt]
\item [(2)]If $\pac$ is locally finite in $\mathbb{C}$ and $\sup_{e\in E}\Theta\leq \pi-\epsilon$, then there is a Euclidean similarity $h$ such that $\mathcal{P}'=h(\mathcal{P}).$
\end{itemize}
\end{theorem}

By Theorems \ref{infinite_existence} and \ref{thm-intro-rigidity-cp}, for a tame ADT $(G,\Theta)$ that satisfies \((Z1)\)-\((Z4)\), there is always an RCP realizing $(G,\Theta)$. In addition, with respect to the circle pattern being in the open unit disk $\D$ or the complex plane $\C$, the circle pattern is rigid up to the corresponding transformations.

A basic feature of an embedded \emph{RCP} is the comparability of radii of neighboring circles. The local Ring Lemma gives such a comparison with a constant depending on the local configuration.

\begin{lemma}[Ring lemma \cite{GeRCP}]\label{ringlemma}
Let $\mathcal{T}= (V, E, F)$ be a \emph{finite} triangulation of the closed disk $\overline{\mathbb{D}}$, and let $\Theta\in[0 ,\pi)^E$ be an angle function.
Let $G=(V,E)$ be its underlying graph.
Assume conditions ($Z_1$) and ($Z_2$) hold, and there exists an embedded circle pattern $\pac = \{C_i\}_{i \in V}$ realizing $(G, \Theta)$ weakly. Take a constant $\epsilon>0$ so that $\Theta([v_i,v_j])\leq\pi-\epsilon$ for all $[i,j]\in E$ (since $\mathcal{T}$ is finite, such $\epsilon$ can always be taken). Then for each vertex $i\in V$ with $B(i, \frac{2\pi}{\epsilon})\cap \partial V= \emptyset$, there exists a constant $C= C(G, \Theta)>0$ such that
\begin{equation}
	\frac{r_j}{r_i}\geq C, \quad \forall j \sim i.
\end{equation}
\end{lemma}

A refined ring lemma will be proved in Section~\ref{sec:refined}. It gives a quantitative lower bound in terms of the flower degree.

\begin{lemma}[Refined ring lemma for RCP]
\label{rringlemma}
Let $(G,\Theta)$ be a tame ADT satisfying (Z1)--(Z4), and assume that (Z2) is strengthened to (Z2'). Let $r:V\to(0,\infty)$ denote the radius function of an embedding in the plane. Then there exists a constant $C=C(\varepsilon_0,\varepsilon_1)$, where $\varepsilon_0$ is the constant in (Z2') and $\varepsilon_1$ is the tameness constant in the angle upper bound $\Theta(e)\leq \pi-\varepsilon_1$, such that for every edge $u\sim v$,
\[
\frac{r(v)}{r(u)} > e^{-C\cdot S(u)},
\]
where $S(u)=\sum_{v'\sim u}\deg(v')$ is the \emph{flower degree} of $u$.
\end{lemma}

 \subsection{Infinite trivalent hyperbolic polyhedra in \texorpdfstring{$\mathbb H^3$}{H3}}

\begin{definition}[Infinite trivalent hyperbolic polyhedron]
An infinite trivalent hyperbolic polyhedron (THP) is a convex
hyperbolic polyhedron
\[
    P=\bigcap_{i\in V}H_i\subset \mathbb H^3
\]
with infinitely many supporting half-spaces, such that every vertex is
incident to exactly three faces, where ordinary, ideal, and hyperideal vertices are permitted.
\end{definition}

\begin{proposition}[\cite{GeRCP}]\label{prop:thp-z12}
    Let $P = \bigcap_{i \in V} H_i$ be a THP with ADT $(G,\Theta)$. Then conditions (Z1) and (Z2) hold for $(G,\Theta)$.
\end{proposition}

\begin{theorem}[RCP--THP correspondence \cite{GeRCP}]
Let \(\mathcal P=\{C_i\}_{i\in V}\) be an RCP with carrier graph
\(G=(V,E)\), and let \(\Theta:E\to (0,\pi)\) be its intersection angle
function. If \((G,\Theta)\) satisfies \((Z1)\) and \((Z2)\), then
\[
    P(\mathcal P)=\bigcap_{i\in V}H_i
\]
is an infinite trivalent hyperbolic polyhedron. It is combinatorially
isomorphic to the Poincar\'e dual of \(G\), and its dihedral angles are
given by
\[
    \Theta_P(e^*)=\Theta(e),\qquad e\in E.
\]
\end{theorem}

There are two possible conformal types for THPs.

\begin{definition}[Parabolic and hyperbolic THP]
Let \(P\) be an infinite THP, and let \(P^{\rm trun}\) be the polyhedron
obtained by truncating all ideal and hyperideal vertices. Let
\(\{K_n\}_{n\geq 1}\) be an exhaustion of \(\mathbb H^3\) by compact sets.

\begin{enumerate}[(1)]
    \item \(P\) is called \emph{parabolic} if
    \(P^{\rm trun}\setminus K_n\) tends to a single point in
    \(\partial_\infty\mathbb H^3\). Equivalently, the set of accumulation
    points of the boundary faces of \(P\) consists of one point.

    \item \(P\) is called \emph{hyperbolic} if
    \(P^{\rm trun}\setminus K_n\) tends to a half-sphere in
    \(\partial_\infty\mathbb H^3\). Equivalently, the accumulation points of
    the boundary faces form a circle in \(\partial_\infty\mathbb H^3\), and
    all faces of \(P\) lie in one of the two half-spaces bounded by the
    hyperbolic plane determined by this circle.
\end{enumerate}
\end{definition}

This definition is independent of the chosen exhaustion. On the circle
pattern side, the corresponding dichotomy is the usual RCP-parabolic/RCP-hyperbolic dichotomy: an angled disk triangulation graph is RCP-parabolic
if it is realized by an RCP with carrier \(\mathbb C\), and RCP-hyperbolic
if it is realized by an RCP with carrier \(\mathbb D\). The uniformization
theorem for RCPs and THPs says that, under the angle conditions above and
the assumption \(\sup_{e\in E}\Theta(e)<\pi\), this analytic type agrees
with the VEL type of the underlying disk triangulation graph; with positive
angles, it also agrees with the parabolic/hyperbolic type of the associated
THP.

\subsection{Geometric characteristic number}

We now introduce the geometric characteristic number for THP, following Ge-Lin \cite{GeLin2024}.

Let \(P\subset \mathbb H^3\) be a trivalent hyperbolic polyhedron. For a vertex \(v \), there are exactly three faces
incident to \(v\). Denote them by
\[
    f_1,\ f_2,\ f_3,
\]
 Write
\[
    e_{ij}=f_i\cap f_j,\qquad 1\leq i<j\leq 3,
\]
and let
\[
    \Theta_{ij}\in(0,\pi)
\]
be the dihedral angle of \(P\) along \(e_{ij}\).

We associate to the pair \((v,f)\) the angle
\[
\theta_v^f
=
\arccos\!\left(
\frac{
1+\cos\Theta_{12}+\cos\Theta_{13}-\cos\Theta_{23}
}{
2\sqrt{1+\cos\Theta_{12}}\sqrt{1+\cos\Theta_{13}}
}
\right).
\]
Equivalently, \(\theta_v^f\) is the Euclidean angle at the vertex
corresponding to \(f_1\) in the auxiliary Euclidean triangle whose side
lengths are
\[
    l_{ij}=\sqrt{2+2\cos\Theta_{ij}},
    \qquad 1\leq i<j\leq 3.
\]
In particular, if \(v\) is incident to the three faces
\(f_1,f_2,f_3\), then the three associated angles satisfy
\begin{equation}\label{angle sum}
    \theta_v^{f_1}+\theta_v^{f_2}+\theta_v^{f_3}=\pi.
\end{equation}
Indeed, they are precisely the three interior angles of the above auxiliary
Euclidean triangle.

\begin{definition}[Geometric characteristic number of a THP]
Let \(P\) be a trivalent hyperbolic polyhedron. The geometric characteristic
number of a face \(f\in F(P)\) is defined by
\[
    L_f(P)
    =
    2\pi-\sum_{v\in f}\theta_v^f .
\]
The geometric characteristic number of \(P\) is the collection
\[
    L(P)=\bigl(L_f(P): f\in F(P)\bigr).
\]
\end{definition}

In the special case where all dihedral angles are zero, the auxiliary
Euclidean triangles are equilateral. Hence each local angle is \(\pi/3\),
and the above definition gives
\[
    L_f(P)
    =
    2\pi-\frac{\pi}{3}\deg(f),
\]
where \(\deg(f)\) denotes the number of vertices of the face \(f\).
Consequently, in the unimodular setting,
\[
    \mathbb E[L_{f_\rho}(P)]=0
    \quad\Longleftrightarrow\quad
    \mathbb E[\deg(f_\rho)]=6,
\]
in the unweighted case.

In our previous work \cite{IIP2026}, the geometric characteristic number at face $f$ of an
ideal hyperbolic polyhedron $\mathcal{P}$ (not necessarily trivalent) with dihedral angles $\Theta_{\mathcal{P}}$ was
\[
T_P(f)=2\pi-\sum_{e\subset\partial f}\Theta_P(e).
\]
The following proposition shows that  \(L_f(P)\) and \(T_{P}(f)\) introduced in \cite{IIP2026} are compatible on their common domain of definition.
Thus, on the class of trivalent ideal hyperbolic polyhedra, the
parabolic--hyperbolic dichotomies in the two settings agree exactly.

\begin{proposition}[Compatibility with the ideal case]
\label{prop:compatibility-ideal}
Let \(P\) be a trivalent ideal hyperbolic polyhedron with dihedral angles $\Theta_{\mathcal{P}}$, and let \(f\) be a
face of \(P\).  Then
\[
    L_f(P)=T_P(f).
\]
More precisely, if \(v\in f\) is a vertex and the two edges of \(f\)
incident to \(v\) have dihedral angles \(A\) and \(B\), then
\[
    \theta_v^f=\frac{A+B}{2}.
\]
\end{proposition}

\begin{proof}
Let \(v\) be a vertex of \(P\), and denote the three faces incident to
\(v\) by \(f_1,f_2,f_3\). Write
\[
    A=\Theta_{12},
    \qquad
    B=\Theta_{13},
    \qquad
    C=\Theta_{23},
\]
and suppose that \(f=f_1\). Since \(v\) is an ideal trivalent vertex, the
three dihedral angles satisfy
\[
    A+B+C=\pi.
\]
By the definition of the local angle \(\theta_v^{f_1}\),
\[
    \cos\theta_v^{f_1}
    =
    \frac{1+\cos A+\cos B-\cos C}
    {2\sqrt{1+\cos A}\sqrt{1+\cos B}}.
\]
Since \(C=\pi-A-B\), we have
\[
    -\cos C=\cos(A+B),
\]
and hence
\[
    1+\cos A+\cos B-\cos C
    =
    1+\cos A+\cos B+\cos(A+B).
\]
Using the elementary identity
\[
    1+\cos A+\cos B+\cos(A+B)
    =
    4\cos\frac{A}{2}
     \cos\frac{B}{2}
     \cos\frac{A+B}{2},
\]

we obtain
\[
    \cos\theta_v^{f_1}
    =
    \cos\frac{A+B}{2}.
\]
Since \(0<A+B<\pi\), it follows that
\[
    \theta_v^{f_1}
    =
    \frac{A+B}{2}.
\]

Now let the boundary edges of \(f\) be
\[
    e_1,\ldots,e_n
\]
in cyclic order, with corresponding dihedral angles
\[
    \Theta_1,\ldots,\Theta_n.
\]
Let
\[
    v_i=e_{i-1}\cap e_i,
\]
where the indices are understood cyclically. The preceding local identity
gives
\[
    \theta_{v_i}^{f}
    =
    \frac{\Theta_{i-1}+\Theta_i}{2}.
\]
Consequently,
\begin{align*}
    \sum_{v\in f}\theta_v^f
    &=
    \frac{1}{2}
    \sum_{i=1}^{n}
    \bigl(\Theta_{i-1}+\Theta_i\bigr) \\
    &=
    \sum_{i=1}^{n}\Theta_i \\
    &=
    2\pi-T_P(f).
\end{align*}
Therefore,
\[
    L_f(P)
    =
    2\pi-\sum_{v\in f}\theta_v^f
    =
    T_P(f).
\]
\end{proof}

\subsection{Unimodularity and mass transport}

A rooted graph is a connected, locally finite graph \(G=(V,E)\) with a
distinguished vertex \(\rho\in V\). The space of rooted locally finite
graphs is equipped with the local (Benjamini-Schramm) topology.
Let \(\mathcal{G}_{\bullet}\) denote the space of isomorphism classes of
rooted locally finite graphs, and let \(\mathcal{G}_{\bullet\bullet}\)
denote the space of isomorphism classes of locally finite graphs with an
ordered pair of distinguished vertices. Following Aldous-Lyons
\cite{AldousLyons2007}, a random rooted graph \((G,\rho)\) is called
\emph{unimodular} if it satisfies the \emph{mass transport principle}:
for every nonnegative Borel function
\[
f:\mathcal{G}_{\bullet\bullet}\to[0,\infty),
\]
one has
\[
\mathbb{E}\Big[\sum_{v\in V(G)} f(G,\rho,v)\Big]
=
\mathbb{E}\Big[\sum_{v\in V(G)} f(G,v,\rho)\Big].
\]
Equivalently, the expected mass sent from the root equals the expected mass
received at the root \cite{AldousLyons2007,BenjaminiSchramm2001}.

Unimodularity can be regarded as the distributional analogue of choosing a root uniformly at random in a finite graph \cite{AldousLyons2007,BenjaminiSchramm2001,BenjaminiCurien2012,Curien2018}.
Many natural random planar graphs and maps are unimodular, including local weak limits of finite planar graphs \cite{BenjaminiSchramm2001}, the uniform infinite planar triangulation
(UIPT) \cite{AngelSchramm2003}, the uniform infinite planar quadrangulation (UIPQ), and planar graphs
arising from stationary point processes.

\subsection{Vertex extremal length and invariant amenability}

Vertex extremal length is a discrete conformal invariant for infinite
graphs and is useful in the study of disk triangulations and circle
packings \cite{HeSchramm1995}. Let \(G=(V,E)\) be an infinite, locally
finite graph. A \emph{vertex metric} is a function
\[
m:V\to[0,\infty).
\]
For a vertex path \(\gamma=(v_0,v_1,\ldots)\), define
\[
\ell_m(\gamma)=\sum_k m(v_k),
\qquad
\operatorname{area}(m)=\sum_{v\in V}m(v)^2.
\]
Fix a root \(\rho\), and let \(\Gamma(\rho\to\infty)\) be the family of
all infinite simple paths starting from \(\rho\). The vertex extremal
length from \(\rho\) to infinity is
\[
\operatorname{VEL}_G(\rho\to\infty)
=
\sup_m
\frac{
\left(\inf_{\gamma\in\Gamma(\rho\to\infty)}\ell_m(\gamma)\right)^2
}
{\operatorname{area}(m)},
\]
the supremum is over all vertex metrics \(m\) with
\(0<\operatorname{area}(m)<\infty\). The graph \(G\) is called
\emph{VEL-parabolic} if
\[
\operatorname{VEL}_G(\rho\to\infty)=\infty,
\]
and \emph{VEL-hyperbolic} otherwise.

For disk triangulations, VEL type is closely related to circle packing
type. In particular, the parabolic case corresponds to circle packing in
\(\mathbb C\), while the hyperbolic case corresponds to circle packing in
\(\mathbb D\) \cite{HeSchramm1995}. For bounded degree planar
triangulations, this is also related to recurrence and transience of
simple random walk \cite{DoyleSnell1984,GGN2013,HeSchramm1995,Soardi1994}.
The circle patterns also connect harmonic functions with the conformal type of the patterns
\cite{BenjaminiSchramm1996Harmonic}.

For unimodular random rooted graphs, invariant amenability is the
unimodular version of amenability \cite{AldousLyons2007}. For one-ended
unimodular random triangulations, Angel-Hutchcroft-Nachmias-Ray proved
a parabolic--hyperbolic dichotomy relating circle packing type, VEL type,
invariant amenability, and random walk behavior
\cite{angel2016unimodular,angel2018hyperbolic}. More precisely, the
parabolic case corresponds to invariant amenability, while the hyperbolic
case corresponds to invariant non-amenability. The hyperbolic case is transient without any bounded-degree assumption, whereas recurrence in the parabolic case requires an additional bounded-degree assumption; see \cite{angel2018hyperbolic}.

\section{Unimodular Gauss-Bonnet formula and dichotomy}
\label{sec:dichotomy}
In this section, we prove the unimodular Gauss-Bonnet formula and
derive the ADT and THP dichotomies. Theorem~\ref{thm:dichotomy} then follows from Theorems~\ref{thm:chr2deg} and~\ref{thm:dichotomya}. 

\begin{theorem}\label{thm:chr2deg}
Let $P$ be a unimodular, locally finite tame random THP satisfying (Z3) and (Z4) and rooted at face $f$. If $E[\deg(f)]< \infty$, then
\[
    \mathbb E[L_f(P)]
    =
    2\pi-\frac{\pi}{3}\mathbb E[\deg(f)],
\]
where \(\deg(f)\) denotes the number of vertices of the face \(f\).

\end{theorem}

\begin{proof}
    Let \((G,\rho,\Theta)\) be the dual ADT of \(P\), where the root vertex \(\rho\) corresponds to the root face \(f\) of \(P\). 
    Through the THP/ADT duality, it is equivalent to proving the following ADT formula for unimodular, locally finite, tame
    ADTs satisfying \((Z1)\)--\((Z4)\):
    \[
    \mathbb{E}[L_\rho(G,\Theta)]
    =
    2\pi-\frac{\pi}{3}\mathbb{E}[\deg_G(\rho)].
    \]
Let \[\theta_u^f = \theta_u^{vw} = \text{arccos}\left( \frac{1+\cos\Theta_{uv} + \cos\Theta_{uw} - \cos\Theta_{vw}}{2 \sqrt{1+\cos\Theta_{uv}}\sqrt{1+\cos\Theta_{uw}}} \right) ,\] where \(f\) is the face defined by the points \(u, v\text{ and } w\). This is the summand in the definition of the geometric characteristic number
 $L(G,\Theta)$.
We define a mass transport function as follows:

\[
m(u, v) = 
\begin{cases}
2\pi - L_u(G, \Theta),&u=v,\\
\theta_u^{vw_1} + \theta_u^{vw_2}, &u\sim v,\\
0, &otherwise,
\end{cases}
\]
where \(uvw_1\) and \(uvw_2\) are the two faces that contain \(uv\) as an edge when \(u\sim v\).

\begin{figure}
    \centering
    \includegraphics[width=0.6\linewidth]{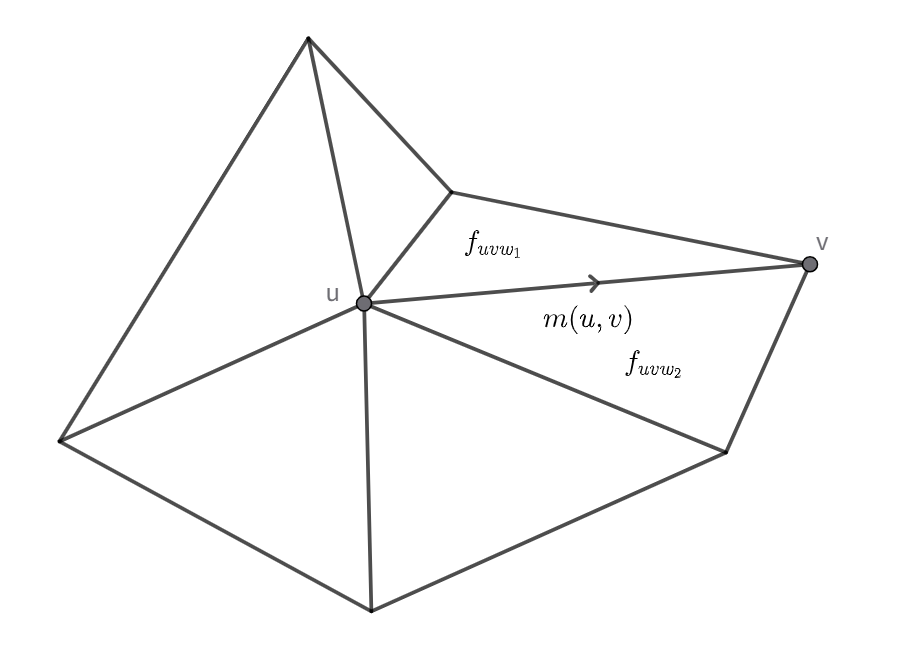}
    \caption{The mass transport corresponding to adjacent faces}
    \label{fig:placeholder}
\end{figure}

Then we calculate the mass received and sent at each vertex \(u\). Since the graph is locally finite, all the sums involved are finite sums. First, the mass sent:

\begin{align*}
\sum_{v \in V(G)} m(u, v)
&= 
\sum_{u\sim v} m(u, v) + 2\pi-L_u(G, \Theta)\\ 
&= 
\sum_{u\sim v} (\theta_u^{vw_1} + \theta_u^{vw_2}) + 2\pi-L_u(G, \Theta)\\
&=
2\sum_{f \ni u} \theta_u^f +2\pi- L_u(G, \Theta)\\
&=
6\pi-3 L_u(G, \Theta),
\end{align*}
Therefore the mass sent from the root \(\rho\) is
\[6\pi - 3L_\rho(G,\Theta).\]
And the mass received:
\begin{align*}
\sum_{v \in V(G)} m(v, u)
&=
\sum_{u\sim v} m(v, u) + 2\pi - L_u(G, \Theta)\\
&=
\sum_{u\sim v} (\theta_v^{uw_1} + \theta_v^{uw_2}) + \sum_{f \in u} \theta_u^f\\
&=
\sum_{f \ni u} (\theta_v^f + \theta_w^f + \theta_u^f) && \text{($u,v,w$ are vertices of face $f$)}\\
&=
\sum_{f \ni u} \pi && (\text{By\ }~\eqref{angle sum})\\
&=
\pi \deg_G(u),
\end{align*}
Then the mass received at root \(\rho\) is 
\[\pi \deg_G(\rho).\]
By the mass transport principle, the expected mass sent equals the expected mass received at the root. Therefore, we have the following equation:
\[\mathbb{E}[L_\rho(G, \Theta)] =  2\pi - \frac{\pi}{3}\mathbb{E}[\deg_G(\rho)].\]
\end{proof}

\begin{remark}
 The term $m(u,u)$ was used to simplify the calculation.
\end{remark}

\begin{remark}
Although \(L_\rho(G,\Theta)\) depends locally on the angle \(\Theta\),
its expectation depends only on \(\mathbb E[\deg_G(\rho)]\).  This follows
from the mass transport principle and recovers the usual degree-six
threshold from the unweighted case.
\end{remark}


\begin{theorem}[Dichotomy]\label{thm:dichotomya}
    Let $(G,\rho,\Theta)$ be an infinite, locally finite, simple, 
    one-ended, ergodic, tame unimodular random ADT satisfying 
    \emph{(Z1)--(Z4)}. If $\mathbb{E}[\deg(\rho)]<\infty$, the following assertions are equivalent:
\begin{enumerate}
\item[\rm(i)] $\mathbb{E}[L_{\rho}(G,\Theta)] = 0$ (respectively, $\mathbb{E}[L_{\rho}(G,\Theta)] < 0$);
\item[(ii)] $\mathbb{E}[\deg_G(\rho)] = 6$ (respectively, $\mathbb{E}[\deg_G(\rho)] > 6$);
\item[\rm(iii)] $G$ is almost surely VEL-parabolic (respectively, VEL-hyperbolic);
\item[\rm(iv)] $G$ is almost surely RCP-parabolic (respectively, RCP-hyperbolic);
\item[\rm(v)] $G$ is almost surely invariantly amenable (respectively, invariantly non-amenable).
\end{enumerate}

Moreover, in the parabolic case, if the degrees are uniformly bounded, then \(G\) is
almost surely recurrent. In the hyperbolic case, \(G\) is almost surely transient. 

\end{theorem}

\begin{proof}

$\rm(i) \leftrightarrow \rm(ii)$: By Theorem~\ref{thm:chr2deg} the condition \(\mathbb{E}[L_\rho(G, \Theta)] = 0\) is equivalent to \(\mathbb{E}[\deg_G(\rho)] = 6\);  \(\mathbb{E}[L_\rho(G, \Theta)] < 0\) is equivalent to \(\mathbb{E}[\deg_G(\rho)] > 6\). 

$\rm(ii) \leftrightarrow \rm(iii)$ and $\rm(ii) \leftrightarrow \rm(v)$: By the dichotomy theorem for unimodular planar maps \cite{angel2018hyperbolic}, this provides equivalence of (ii) with (iii) and (v).

$\rm(iii) \leftrightarrow \rm(iv)$: Given properties (Z1) - (Z4), we know there is a rigid regular circle pattern realizing the angles by Theorems~\ref{infinite_existence} and~\ref{thm-intro-rigidity-cp}, and by \cite[Theorem 1.4]{GeRCP}, we know (iii) and (iv) are equivalent for \((G, \Theta)\). Therefore the five criteria (i) to (v) are equivalent.

Finally, in the hyperbolic case, \(G\) is VEL-hyperbolic by
\({\rm (iii)}\), and hence transient by \cite[Lemma~3.14]{angel2018hyperbolic}. In the parabolic case, if the degrees are uniformly bounded, then VEL-parabolicity implies recurrence; see \cite[Theorem~1.4]{GeRCP}.
\end{proof}

\begin{corollary}
   Let \((G,\rho,\Theta)\) be an infinite, locally finite, simple,
one-ended, ergodic, tame unimodular random ADT satisfying
\emph{(Z1)--(Z4)} and $E[\deg(\rho)]< \infty$. If the mark $\Theta$ is constant, and $\E[\deg_G(\rho)] > 6$, then $G$ is a.s. RCP-hyperbolic.
\end{corollary}

\begin{proof}
    By the definition of the geometric characteristic number:
    $$L_\rho(G, \Theta) = 2\pi - \sum_{f \ni \rho} \theta_\rho^f.$$
Since $\Theta$ is constant, all $\theta$ angles equal $\pi/3$ by definition. Therefore 
    $$\E[L_\rho(G, \Theta)] = 2\pi - \frac{\pi}{3}\E[\deg_G(\rho)] < 0,$$
    and we finish the proof by Theorem~\ref{thm:dichotomy}.
\end{proof}

At the end of this section, we prove Theorem~\ref{thm:polyhedral-dichotomy}.

\begin{proof}
Let $(G, \rho, \Theta)$ be the dual rooted ADT of $P$. By Proposition \ref{prop:thp-z12}, conditions (Z1) and (Z2) hold automatically, while (Z3) and (Z4) are assumed. Under the THP/ADT correspondence, the parabolic/hyperbolic type of $P$ agrees with the RCP type of $(G, \Theta)$, and $L_f(P)=L_\rho(G, \Theta)$. The conclusion therefore follows directly from Theorem \ref{thm:dichotomy}.
\end{proof}


\section{Polyhedrally sofic THPs}
\label{sec:polyhedrally-sofic}

We now study which unimodular infinite trivalent hyperbolic polyhedra
can arise as local limits of finite trivalent hyperbolic polyhedra.

Let \(P\) be a finite trivalent hyperbolic polyhedron, and let \(f\) be
a face chosen uniformly from \(F(P)\). Denote by
\[
    (G(P),\rho,\Theta_P)
\]
the finite rooted marked dual triangulation of $P$, where
\(\rho\) is the dual vertex corresponding to \(f\).

\begin{definition}[Polyhedrally sofic THP]
\label{def:polyhedrally-sofic}
A unimodular random rooted infinite trivalent hyperbolic polyhedron
\((P,f)\) is called \emph{polyhedrally sofic} if there exists a
sequence of finite trivalent hyperbolic polyhedra \(P_n\), with
\[
    |F(P_n)|\longrightarrow\infty,
\]
such that, after choosing \(f_n\) uniformly from \(F(P_n)\), the finite rooted marked dual triangulation of $P_n$ converge
\[
    \bigl(G(P_n),\rho_n,\Theta_{P_n}\bigr)
    \xrightarrow[n\to\infty]{}
    \bigl(G(P),\rho,\Theta_P\bigr)
\]
in the local topology of rooted marked graphs.

Equivalently, \((P,f)\) is polyhedrally sofic if its ADT is a Benjamini--Schramm limit of the finite rooted marked dual triangulations of uniformly
face-rooted finite trivalent hyperbolic polyhedra.
\end{definition}

Thus polyhedral soficity concerns the full marked polyhedral structure,
including both the local combinatorics and the dihedral-angle data. It is
strictly more restrictive than asking only that the underlying rooted
graph be a local weak limit of finite graphs.

The main ingredient is the following finite counterpart of the
unimodular Gauss-Bonnet formula.

\begin{proposition}[Finite polyhedral Gauss-Bonnet formula]
\label{prop:finite-polyhedral-gauss-bonnet}
Let \(P\) be a finite trivalent hyperbolic polyhedron. Then
\[
    \sum_{f\in F(P)}L_f(P)=4\pi.
\]
Consequently, if \(f\) is chosen uniformly from \(F(P)\), then
\[
    \mathbb E[L_f(P)]
    =
    \frac{4\pi}{|F(P)|}.
\]
Moreover, if \(\rho\) is the corresponding uniformly chosen vertex of
the dual triangulation \(G(P)\), then
\[
    \mathbb E[\deg_{G(P)}(\rho)]
    =
    6-\frac{12}{|F(P)|}.
\]
\end{proposition}

\begin{proof}
This is a classical result and we only outline the argument.

\noindent By definition,
\[
\sum_{f\in F(P)} L_f(P)
=
2\pi |F(P)|
-
\sum_{f\in F(P)}\sum_{v\in f}\theta_v^f.
\]
At each trivalent vertex, the three incident angles satisfy
\[
\theta_v^{f_1}+\theta_v^{f_2}+\theta_v^{f_3}=\pi.
\]
Summing this over all vertices and applying Euler's formula together with the trivalence condition immediately yields
\[
\sum_{f\in F(P)} L_f(P)=4\pi.
\]
The second assertion follows by averaging over a uniformly chosen face. For the last assertion, observe that the vertices and edges of the dual triangulation $G(P)$ are respectively the faces and edges of $P$; the desired formula then follows again from Euler's formula.
\end{proof}



The finite identity is independent of the dihedral angles. All angular
dependence cancels after summation, and the total characteristic \(4\pi\)
is determined solely by the spherical topology of the boundary of a
finite convex polyhedron.

We now show that this topological constraint forces every admissible
infinite polyhedral local limit to be parabolic.

\begin{theorem}[Benjamini--Schramm parabolicity]
\label{thm:polyhedral-bs-parabolicity}
Let \(P_n\) be a sequence of finite trivalent hyperbolic polyhedra, and
let \(f_n\) be chosen uniformly from \(F(P_n)\). Assume that
\[
    |F(P_n)|\longrightarrow\infty
\]
and that the finite rooted marked dual triangulations converge locally in
distribution:
\[
    \bigl(G(P_n),\rho_n,\Theta_{P_n}\bigr)
    \xrightarrow[n\to\infty]{\mathrm{law}}
    (G,\rho,\Theta).
\]
Suppose that \((G,\rho,\Theta)\) is almost surely infinite, locally
finite, simple, one-ended, tame, and satisfies
\emph{(Z1)--(Z4)}. Then \((G,\rho,\Theta)\) is unimodular and
\[
    \mathbb E[\deg_G(\rho)]=6,
    \qquad
    \mathbb E[L_\rho(G,\Theta)]=0.
\]
Moreover, every ergodic component of its law is RCP-parabolic and
invariantly amenable. Consequently, the infinite trivalent hyperbolic
polyhedron associated with the limiting RCP is almost surely
parabolic.

No moment or uniform-integrability assumption on the degrees is
required.
\end{theorem}

\begin{proof}
Each finite marked graph rooted at a uniformly chosen vertex is
unimodular, and unimodularity is preserved under local weak convergence
\cite{AldousLyons2007}. Hence the limiting rooted marked
triangulation \((G,\rho,\Theta)\) is unimodular.

Set
\[
    D_n:=\deg_{G(P_n)}(\rho_n),
    \quad
    D:=\deg_G(\rho).
\]
For every \(M<\infty\), the truncated degree \(D\wedge M\) is a bounded
local observable. Local convergence therefore gives
\[
    \mathbb E[D\wedge M]
    =
    \lim_{n\to\infty}
    \mathbb E[D_n\wedge M].
\]
Since \(D_n\wedge M\leq D_n\), Proposition
\ref{prop:finite-polyhedral-gauss-bonnet} implies
\[
\begin{aligned}
    \mathbb E[D\wedge M]
    &\leq
    \lim_{n\to\infty}\mathbb E[D_n] \\
    &=
    \lim_{n\to\infty}
    \left(
        6-\frac{12}{|F(P_n)|}
    \right)
    =
    6.
\end{aligned}
\]
Letting \(M\to\infty\) and applying the monotone convergence theorem
yields
\begin{equation}\label{eq:bs-degree-upper}
    \mathbb E[D]\leq6.
\end{equation}

Let
\[
    \mu=\int \mu_\xi\,d\lambda(\xi)
\]
be the ergodic decomposition of the law of
\((G,\rho,\Theta)\), and set
\[
    m(\xi):=\mathbb E_{\mu_\xi}[D].
\]
Since \(\mathbb E_\mu[D]\leq6\), we have \(m(\xi)<\infty\) for
\(\lambda\)-almost every \(\xi\). The ADT dichotomy,
Theorem~\ref{thm:dichotomy}, applies to each such ergodic component.
Consequently, for almost every \(\xi\), either
\[
    m(\xi)=6
    \quad\text{and}\quad
    \mu_\xi\text{ is RCP-parabolic},
\]
or
\[
    m(\xi)>6
    \quad\text{and}\quad
    \mu_\xi\text{ is RCP-hyperbolic}.
\]
In particular,
\[
    m(\xi)\geq6
\]
for \(\lambda\)-almost every \(\xi\). Integrating over the ergodic
decomposition gives
\begin{equation}\label{eq:bs-degree-lower}
    \mathbb E_\mu[D]
    =
    \int m(\xi)\,d\lambda(\xi)
    \geq6.
\end{equation}
Combining \eqref{eq:bs-degree-upper} and
\eqref{eq:bs-degree-lower}, we obtain
\[
    \mathbb E[D]=6.
\]
Since \(m(\xi)\geq6\) almost surely and its average is exactly \(6\),
we must have
\[
    m(\xi)=6
\]
for almost every ergodic component. Theorem~\ref{thm:dichotomy} then
implies that every ergodic component is RCP-parabolic and invariantly
amenable.

Finally, the unimodular Gauss-Bonnet formula,
Theorem~\ref{thm:chr2deg}, gives
\[
\begin{aligned}
    \mathbb E[L_\rho(G,\Theta)]
    &=
    2\pi-\frac{\pi}{3}\mathbb E[D]=0.
\end{aligned}
\]
Through the RCP--THP correspondence, RCP-parabolicity is equivalent to
parabolicity of the associated infinite trivalent hyperbolic
polyhedron.
\end{proof}

\begin{remark}
Local convergence alone does not imply convergence of the unbounded
random variables \(D_n\). The preceding proof only uses the bounded
local observables \(D_n\wedge M\), which yield the upper bound
\[
    \mathbb E[\deg_G(\rho)]\leq6.
\]
The missing reverse inequality is supplied by the unimodular
parabolic--hyperbolic dichotomy on the ergodic components of the
limiting law.
\end{remark}

As an immediate consequence, the hyperbolic regime cannot be obtained
from finite trivalent hyperbolic polyhedra.

\begin{corollary}[No hyperbolic polyhedral local limits]
\label{cor:hyperbolic-not-polyhedrally-sofic}
Let \((P,f)\) be a tame unimodular random infinite trivalent hyperbolic
polyhedron whose dual ADT is almost surely infinite, locally finite,
simple, one-ended, and satisfies \emph{(Z1)--(Z4)}. If
\[
    \mathbb P(P\text{ is hyperbolic})>0,
\]
then \((P,f)\) is not polyhedrally sofic.

In particular, every ergodic hyperbolic unimodular THP satisfying these
assumptions is not polyhedrally sofic.
\end{corollary}

\begin{proof}
If \((P,f)\) were polyhedrally sofic, then its dual rooted ADT
would be a Benjamini--Schramm limit of the rooted marked dual triangulations of uniformly face-rooted finite
trivalent hyperbolic polyhedra. Theorem
\ref{thm:polyhedral-bs-parabolicity} would then imply that \(P\) is
parabolic almost surely, contradicting
\[
    \mathbb P(P\text{ is hyperbolic})>0.
\]
\end{proof}

It is natural to ask whether the converse also holds.
\begin{conjecture}\label{conj}
Let $(P,f)$ be an ergodic, tame, unimodular random rooted infinite trivalent hyperbolic polyhedron whose dual ADT is almost surely infinite, locally finite, simple, one-ended, and satisfies \emph{(Z1)--(Z4)}. Then
\[
    P \text{ is almost surely parabolic}
    \quad\Longrightarrow\quad
    (P,f)\text{ is polyhedrally sofic}.
\]
\end{conjecture}

This does not contradict the Aldous--Lyons conjecture, since polyhedral soficity is substantially stronger than ordinary soficity: the finite
approximations must arise from trivalent hyperbolic polyhedra and must
preserve the dihedral angles. Thus
\[
    \text{polyhedrally sofic}
    \quad\Longrightarrow\quad
    \text{sofic},
\]
The converse does not follow from ordinary soficity alone. Indeed, the underlying graph of a hyperbolic unimodular THP may still
be approximable by triangulations of higher-genus surfaces. For a
triangulation with \(N\) vertices on a closed orientable surface of
genus \(g\),
\[
    \frac{1}{N}\sum_v \deg(v)
    =
    6+\frac{12(g-1)}{N},
\]
so such approximations are not subject to the spherical
Gauss-Bonnet obstruction.

\section{Refined ring lemma for RCP}
\label{sec:refined}
Ring lemmas are a fundamental tool in circle packing theory.  They give
quantitative control on the ratios of neighboring radii and are crucial
for convergence and random-walk estimates.  In the classical
circle packing setting, such estimates go back to the work of Rodin-Sullivan
and have been used systematically in the study of infinite packings and
discrete conformal geometry
\cite{HeSchramm1995,RodinSullivan1987,Stephenson2005}.

In this section we will prove the refined ring lemma for \RCP s.
Compared with the local ring lemma, we quantify the radius ratio by degrees:
for a vertex $v$ with neighbors $v_i$ and $d_i=\deg(v_i)$, the ratio is
controlled by an exponential factor in $\sum_i d_i$.
Compared with the uniform ring lemma, we remove the bounded-degree assumption.

\begin{lemma}[Refined ring lemma for \RCP]\label{thm:re-ring-lemma}
Let $(G,\Theta)$ be a tame ADT satisfying (Z1)--(Z4), and assume that (Z2) is strengthened to (Z2$'$). Let $r:V\to(0,\infty)$ denote the radius function of an embedding in the plane. Then there exists a constant $C=C(\varepsilon_0,\varepsilon_1)$, where $\varepsilon_0$ is the constant in (Z2$'$) and $\varepsilon_1$ is the tameness constant in the angle upper bound $\Theta(e)\leq \pi-\varepsilon_1$, such that for every edge $u\sim v$,
\[
\frac{r(v)}{r(u)} > e^{-C\cdot S(u)},
\]
where $S(u)=\sum_{v'\sim u}\deg(v')$ is the \emph{flower degree} of $u$.
\end{lemma}

Before proving this lemma, we will prove several useful lemmas as in \cite{IIP2026}. In the lemmas and the proof below, \((G, \Theta)\) will be a tame ADT satisfying (Z1)--(Z4), and assume that (Z2) is strengthened to (Z2$'$); $u, v$, and $w$ will denote different vertices (and the corresponding circles in the RCP embedding), $r_u, r_v$, and $r_w$ the corresponding radius.

We note a lemma from previous work:

\begin{lemma}(\cite[Lemma 4.1]{IIP2026})\label{lemma1}
Let $D_1, \dots, D_n$ be disks with centers $O_1, \dots, O_n$ and radii $r_1, \dots, r_n$. Assume that $O_1, \dots, O_n$ form a polygon. $D_i$ and $D_{i+1}$ intersect with an angle $\Theta_i$ (where $1 \le i \le n$ and $D_{n+1} = D_1$). Suppose that:

$$\sum_{i} (\pi - \Theta_i) > 2\pi + \varepsilon_0 .$$

Let $E$ be a subset of the polygon that intersects every disk $D_i$. Then there exists a constant $C = C(\varepsilon_0)$ such that:

$$(\textrm{diam}(E))^{-1/2} < C \sum_{i} r_i^{-1/2}.$$
\end{lemma}

\begin{lemma}\label{lemma2}
There exists $\delta_2 = \delta(\varepsilon_1) < 1$ such that the following holds.  For every vertex $u$ and every neighbor $x\sim u$, if
\[
E_u(x)=\bigcup_{v\sim u,\,v\ne x}(D_v\backslash D_u),
\]
then
\[
\operatorname{diam} E_u(x)\ge\delta_2 r_u .
\]
In particular,
\[
\sum_{v \sim u} r_v - \max_{v \sim u} \{r_v\} \ge \frac{\delta_2}{2} r_u .
\]
\end{lemma}

\begin{proof}
    Since the RCP embedding exists, the local flower of every vertex is a non-degenerate cyclic configuration. In particular, the neighbors of $u$ occur in a genuine cyclic order around the circle of $u$, and degenerate cases with fewer than three complementary neighboring disks cannot occur.

    For each face $f_{uvw}$ that contains $u$ as a vertex, consider the outer intersection point $P_{vw}$ of $v$ and $w$ (by the outer intersection point, we mean the intersection point lying on the line through $v$ and $w$ or on the opposite side of that line from $u$). By the properties (Z1)-(Z4) discussed in \cite{GeRCP}, $P_{vw}$ is outside the circle of $u$.
    
    Fix a neighbor $x\sim u$. Let $y$ be the vertex adjacent to both $x$ and $u$ such that the rotation from $e_{ux}$ to $e_{uy}$ is counterclockwise. Then the counterclockwise angle from $e_{ux}$ to $uP_{xy}$ would be:

    $$
    \angle xuP_{xy} < \angle xuP_{xu} < \pi - \varepsilon_1.
    $$

    With the same arguments, let $z$ be the vertex adjacent to both $x$ and $u$ and the rotation from $e_{ux}$ to $e_{uz}$ be clockwise. Then the clockwise angle from $e_{ux}$ to $uP_{xz}$, $\angle P_{xz}ux < \pi - \varepsilon_1$.

    \begin{figure}
        \centering
        \includegraphics[width=0.6\linewidth]{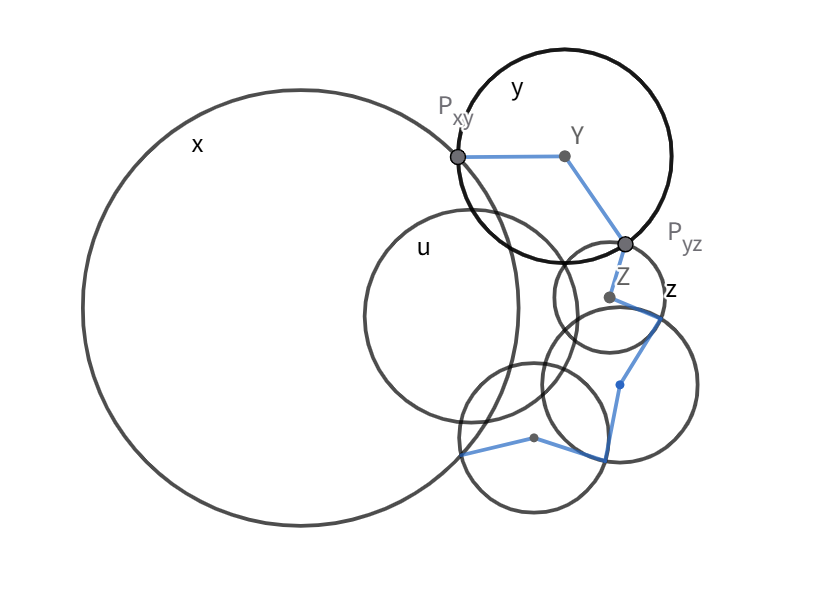}
        \caption{The polyline shown in blue.}
        \label{fig:lemma4.3}
    \end{figure}

 Therefore, the angle from $uP_{xy}$ to $uP_{xz}$ is greater than $2\varepsilon_1$. Since both $P_{xy}$ and $P_{xz}$ belong to $E_u(x)$, we have
    \[
    \operatorname{diam}E_u(x)
    \ge 2r_u\sin(\min\{\varepsilon_1, \frac{\pi}{2}\}).
    \]
    This proves the diameter estimate after decreasing the constant, if necessary, to a number below $1$.

    The final assertion follows from the same cyclic chain. Indeed, if $x$ is chosen so that $r_x=\max_{v\sim u}r_v$, then the poly-line $P_{xy}yw_1\dots w_mzP_{xz}$ has length
    \begin{align*}
        l &= r_y + d_{yw_1} + \dots + d_{w_mz} + r_z\\
        &\le r_y + (r_y + r_{w_1}) + \dots + (r_{w_m} + r_{z}) + r_z \\
        &= 2\left(\sum_{v \sim u} r_v - r_x\right)\\
        &= 2\left(\sum_{v \sim u} r_v - \max_{v \sim u} \{r_v\}\right).
    \end{align*}
    Since this poly-line connects $P_{xy}$ to $P_{xz}$ counterclockwise with respect to $u$, its length is at least $2r_u\sin\varepsilon_1$. Taking, for instance, $\delta_2=\sin\varepsilon_1$ completes the proof.
\end{proof}

\begin{remark}
    The proof differs from \cite[Lemma 4.2]{IIP2026}.
\end{remark}

\begin{lemma}\label{lem:augment}
There exists a constant $\delta_{\mathrm{aug}}=\delta_{\mathrm{aug}}(\varepsilon_1)>0$
such that the following holds. Suppose that $u,a,b$ are pairwise adjacent
in the RCP embedding. Let
\[
E_{a,b}(u)
=
\bigcup_{\substack{v\sim a \text{ or } v\sim b\\ v\ne u,a,b}} D_v .
\]
Then
\[
\operatorname{diam} E_{a,b}(u)
\ge
\delta_{\mathrm{aug}}\max\{r_a,r_b\}.
\]
\end{lemma}

\begin{proof}
We first prove this lemma for the case where $u, a, b$ form a triangular
face in the RCP embedding.

It is enough to prove a lower bound by a constant multiple of $r_a$;
the same argument with $a$ and $b$ interchanged then gives the bound
with $r_b$, and the maximum follows.

Consider the flower of the circle $D_a$.  Since $u,a,b$ are pairwise
adjacent, the two disks $D_u$ and $D_b$ occur consecutively in the
cyclic order of the neighbors of $a$.  Let $x$ be the neighbor of $a$
immediately preceding $u$ in this cyclic order, and let $y$ be the
neighbor of $b$ immediately following $u$.  These vertices are different
from $u,a,b$, and hence the disks $D_x$ and $D_y$ are both contained in
$E_{a,b}(u)$.

For every face $rst$, denote by $P_{st}$ the outer intersection
point of the two neighboring disks $D_s$ and $D_t$, as in the proof of
Lemma~\ref{lemma2}.  By the same local RCP geometry used there, the
points $P_{xu}$ and $P_{uy}$ lie outside the circle $D_a$. Moreover, the
sector at $a$ cut out by the neighbor $u$ has
angular size at most $\pi-\varepsilon_1$. Hence, the complementary angle
from the ray $aP_{xu}$ to the ray $aP_{uy}$, measured through the chain
of neighbors different from $u$ and $b$, is at least $\varepsilon_1$.
The point $P_{xu}$ has a distance of at least $r_a$ from the center of
$D_a$ as in Lemma~\ref{lemma2}. For the point $P_{uy}$, if it lies inside the
disk $D_a$, the intersection graph of centers $a, b, y, u$ would form a planar $K_4$ with $b$ 
in the center. By angle condition (Z2) for the outer triangle $ayu$, all 
intersections of circles $D_y$ and $D_u$ lie outside the disk $D_a$, 
which leads to a contradiction. Therefore, the point $P_{uy}$ has a 
distance of at least $r_a$ from the center of $D_a$, and one gets
\[
 |P_{xu}P_{uy}|
 \ge 2r_a\sin(\varepsilon_1/2).
\]

\begin{figure}
    \centering
    \includegraphics[width=0.5\linewidth]{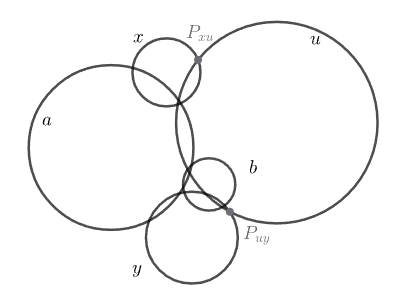}
    \caption{The case where $D_y$ and $D_a$ intersect.}
    \label{fig:lemma4_4}
\end{figure}

The points $P_{xu}$ and $P_{uy}$ belong respectively to $D_x$ and
$D_y$, and therefore both belong to $E_{a,b}(u)$.  Consequently
\[
 \operatorname{diam}E_{a,b}(u)
 \ge 2r_a\sin(\varepsilon_1/2).
\]
Interchanging $a$ and $b$ gives the same estimate with $r_b$ in place
of $r_a$.  Thus the lemma holds with, for instance,
\[
\delta_{\mathrm{aug}}=2\sin(\varepsilon_1/2),
\]
or with this constant decreased if one wants $\delta_{\mathrm{aug}}<1$.

For general cases where $u,a,b$ forms a separating triangle, remove the points and edges in the component enclosed by the separating triangle $uab$. The resulting circle pattern still satisfies the RCP and (Z1)-(Z4) conditions. The preceding argument therefore applies without change.
\end{proof}

\begin{lemma}\label{lemma3}
 Suppose that the neighborhood of $u$ is $N(u)=\{v_1,\dots,v_d\}$, in cyclic order, with indices understood modulo $d$. Then there exists $C_3 = C(\varepsilon_0, \varepsilon_1)$ with the following property. For every non-empty proper cyclic interval
\[
I(a,b)=\{v_{a+1},\ldots,v_{b-1}\},
\qquad a<b<a+d,
\]
with boundary vertices $v_a$ and $v_b$, one has
\[
\left(\sum_{v\in I(a,b)} r_v\right)^{-1/2}
< C_3^{\min(\deg v_a,\deg v_b)}
\left(r_u^{-1/2}+r_{v_a}^{-1/2}+r_{v_b}^{-1/2}\right).
\]
\end{lemma}


\begin{proof}
After replacing the cyclic labeling by an equivalent one if necessary,
we may write the given interval as $I(a,b)=\{v_{i+1},\ldots,v_{j-1}\}$
with $1\leq i<j\leq d$ and boundary vertices $v_i$ and $v_j$.
Put
\[
A=r_u^{-1/2}+r_{v_i}^{-1/2}+r_{v_j}^{-1/2},
\qquad
R_0=\sum_{k=i+1}^{j-1}r_{v_k},
\]
and let $E$ be the union of the disks $v_{i+1},\ldots,v_{j-1}$.
It suffices to prove the estimate with $\deg v_j$ in place of the
minimum, since the same argument applied in the opposite cyclic direction
gives the estimate with $\deg v_i$.

List, in cyclic order around $v_j$, the neighbors of $v_j$ that lie on
the side opposite to $E$ and different from $u$ and $v_i$ as
$w_1,\ldots,w_m$, where $w_1$ is adjacent to $E$ and $w_m$ is adjacent to $u$ or $v_i$. Thus
\[
m\leq \deg v_j-2.
\]
Let $E'$ be the union of $E$ and the circles adjacent to $v_j$ but not in the set $\{u\} \cup \{v_{i+1},\dots,v_{j-1}\} \cup \{w_1,\dots,w_m\}$: $$E' = E \cup \bigcup_{w\sim v_j,w \notin \{u\} \cup \{v_{i+1},\dots,v_{j-1}\} \cup \{w_1,\dots,w_m\}} D_w \backslash D_{v_j} $$. Since $E'$ lies in the convex hull of $E$ (more precisely, the convex hull of disk $D_{v_{j-1}}$ and the disk $D_{v_k}$ of the smallest index that intersects with $D_{v_j}$ in $\{v_{i+1},\dots,v_{j-1}\}$), $\text{diam}(E') = \text{diam}(E) \leq 2R_0$.

For every $1\leq q\leq m$, the polygon bounded by
$v_i,u,v_j,w_q$ in cyclic order satisfies the angle hypothesis of
Lemma~\ref{lemma1} from the condition (Z2'); if the closing edge is not present, the same estimate
follows from \cite[Remark following Lemma 4.1]{IIP2026} using (Z3) and
$\Theta\leq \pi-\varepsilon_1$ since one of the pairs $(uv_i, uv_j)$ and $(uv_j,v_jw_q)$ is homologically non-adjacent. Consequently, there is a constant
$C=C(\varepsilon_0,\varepsilon_1)>1$ such that, for all $q$,
\begin{equation}\label{eq:lemma3-one-step}
\left(R_0+\sum_{\ell=1}^{q-1}r_{w_\ell}\right)^{-1/2}
\leq 2\left(\text{diam} (E') + 2\sum_{\ell=1}^{q-1}r_{w_\ell}\right)^{-1/2}
\leq 2C\left(A+r_{w_q}^{-1/2}\right).
\end{equation}
Choose once and for all
\[
C_3=\max\left\{4C + 1,\frac{64C^2}{\delta_2},\frac{64C^2}{\delta_{\mathrm{aug}}}\right\},
\]
where $\delta_2$ is the constant from Lemma~\ref{lemma2} and $\delta_{\mathrm{aug}}$ is the constant from Lemma~\ref{lem:augment}.  Suppose, toward a
contradiction, that
\[
R_0^{-1/2}\geq C_3^{\deg v_j}A.
\]
We prove by finite induction on $q=0,1,\ldots,m$ that
\begin{equation}\label{eq:lemma3-induction}
\left(R_0+\sum_{\ell=1}^{q}r_{w_\ell}\right)^{-1/2}
\geq C_3^{\deg v_j-q}A.
\end{equation}
The case $q=0$ is the assumption.  If \eqref{eq:lemma3-induction} holds
for $q-1$, then \eqref{eq:lemma3-one-step} gives
\[
r_{w_q}^{-1/2}
\geq (2C)^{-1}\left(R_0+\sum_{\ell=1}^{q-1}r_{w_\ell}\right)^{-1/2}-A
\geq \frac{1}{4C}\left(R_0+\sum_{\ell=1}^{q-1}r_{w_\ell}\right)^{-1/2},
\]
because $C_3\geq 4C$. Hence
\[
r_{w_q}
\leq 16C^2\left(R_0+\sum_{\ell=1}^{q-1}r_{w_\ell}\right)
\leq (C_3^2-1)\left(R_0+\sum_{\ell=1}^{q-1}r_{w_\ell}\right),
\]
and \eqref{eq:lemma3-induction} follows for $q$.  This proves the
finite induction.

Taking $q=m$ and using $m\leq \deg v_j-2$, we obtain
\[
\text{diam}(E')+2\sum_{\ell=1}^{m}r_{w_\ell}\leq 2\left(R_0+\sum_{\ell=1}^{m}r_{w_\ell}\right)\leq 2C_3^{-4}A^{-2}
\leq 2C_3^{-4}r_{v_j}<\min\{\delta_2, \delta_{\mathrm{aug}}\} r_{v_j}.
\]

There are two cases:

If $v_i$ and $v_j$ are adjacent, then the disks counted on the left include all neighbors of $v_j$ and $u$ except $v_i$. This contradicts Lemma~\ref{lem:augment}.

If $v_i$ and $v_j$ are not adjacent, the disks counted on the left include all neighbors of $v_j$ except
possibly the largest one $u$.  This contradicts Lemma~\ref{lemma2}, applied
with the center $v_j$. Hence both cases lead to contradiction.

Therefore,
\[
R_0^{-1/2}<C_3^{\deg v_j}A.
\]
Repeating the same argument with $v_i$ and $v_j$ interchanged gives the
claimed bound with $\min(\deg v_i,\deg v_j)$.
\end{proof}

\begin{proof}[Proof of Lemma~\ref{thm:re-ring-lemma}]

Write the neighbors of $u$ as $N(u)=\{v_1,\ldots,v_d\}$ in cyclic order,
with indices understood modulo $d$.  If $a<b<a+d$, set
\[
I(a,b)=\{v_{a+1},\ldots,v_{b-1}\},\qquad
R(a,b)=\sum_{v\in I(a,b)}r_v,
\]
where $v_{q+d}=v_q$.  Thus $I(a,b)$ is the open cyclic interval from
$v_a$ to $v_b$.

Let
\[
C_0=\max\left\{2,\frac{64C_3^2}{\delta_2},\delta_2^{-1}\right\}.
\]
We claim that no neighbor $v_i$ can satisfy
\[
r_{v_i}\leq C_0^{-S(u)}r_u.
\]
Assume otherwise.  Start with the cyclic interval
$(a_0,b_0)=(i-1,i+1)$, so that $I(a_0,b_0)=\{v_i\}$.  Then
\[
R(a_0,b_0)
\leq C_0^{-S(u)}r_u
\leq C_0^{-S(u)+\sum_{v\in I(a_0,b_0)}\deg v}r_u.
\]
We prove the following finite induction.  Suppose that
\[
R(a,b)\leq C_0^{-S(u)+\sum_{v\in I(a,b)}\deg v}r_u
\]
and that $I(a,b)$ is not yet the complement of a single neighbor of
$u$.  Let the two adjacent outside vertices be $v_a$ and $v_b$.  If both
extensions failed, namely if
\[
R(a-1,b)>C_0^{-S(u)+\sum_{v\in I(a-1,b)}\deg v}r_u
\]
and
\[
R(a,b+1)>C_0^{-S(u)+\sum_{v\in I(a,b+1)}\deg v}r_u,
\]
then the two new boundary radii would satisfy
\[
r_{v_a}> (C_0^{\deg v_a}-1)R(a,b),
\qquad
r_{v_b}> (C_0^{\deg v_b}-1)R(a,b).
\]
On the other hand, Lemma~\ref{lemma3} applied to the interval $I(a,b)$
implies
\[
\min(r_{v_a},r_{v_b})
\leq 16C_3^{2\min(\deg v_a,\deg v_b)}R(a,b),
\]
which contradicts the choice of $C_0$.  Hence at least one of the two
extensions $(a-1,b)$ and $(a,b+1)$ preserves the displayed bound.

Starting from $(a_0,b_0)$ and applying this step exactly $d-2$ times,
we obtain a cyclic interval whose complement consists of one vertex,
say $w$.  Therefore
\[
\sum_{v'\sim u,\,v'\ne w}r_{v'}
\leq C_0^{-S(u)+\sum_{v'\sim u,\,v'\ne w}\deg v'}r_u
=C_0^{-\deg w}r_u
\leq C_0^{-2}r_u
<\delta_2 r_u,
\]
contradicting Lemma~\ref{lemma2} applied with center $u$.  Thus
$r_v>C_0^{-S(u)}r_u$ for every $v\sim u$, and the lemma follows with
$C=\log C_0$.

\end{proof}

\begin{remark}
    If we also assume that (Z3) is strengthened to a uniform condition (Z3'), the ring ratio can be shown to be bounded below by $e^{-C\deg(u)}$ as the circle packing case.
\end{remark}

\section{The probabilistic boundary and face random walk of THP}
\label{sec:boundary}

Based on the refined Ring Lemma for RCP established above, we now identify
the RCP boundary with the Poisson boundary of the face random walk. 
For general background on Poisson boundaries, entropy, and random
walks on graphs and groups, see \cite{KaimanovichVershik1983,Woess2000}.
The proof follows the circle-packing boundary theory for unimodular hyperbolic
triangulations
\cite{angel2016unimodular,angel2018hyperbolic}
and the corresponding ideal-polyhedral argument in \cite{IIP2026}.
We include the main intermediate statements in order to indicate precisely
where the refined Ring Lemma enters.

Assume throughout this section that the unimodular ADT
\((G,\rho,\Theta)\) is simple, one-ended, tame, and RCP-hyperbolic, and
that conditions \({\rm (Z1)}\), \({\rm (Z3)}\), \({\rm (Z4)}\), and the
strengthened cycle condition \({\rm (Z2')}\) hold. Assume further that
\begin{equation}\label{eq:boundary-third-moment0}
    \E[\deg_G(\rho)^3]<\infty.
\end{equation}
For \(u\in V(G)\), let
\[
    S(u):=\sum_{w\sim u}\deg_G(w)
\]
be the flower degree of \(u\).

Let \(C=(z,r)\) be an RCP of \((G,\Theta)\) in \(\D\), where \(z(v)\)
and \(r(v)\) denote the Euclidean center and radius of the circle
corresponding to \(v\). We write \(z_h(v)\) and \(r_h(v)\) for the
corresponding hyperbolic center and hyperbolic radius. Let
\((X_n)_{n\geq0}\) be simple random walk on \(G\), started from
\(X_0=\rho\). Under the THP/ADT duality, this is precisely the face random
walk on the corresponding trivalent hyperbolic polyhedron.

For a fixed realization of \((G,\rho,\Theta)\), we denote by
\(\mathbb P_v^G\) and \(\mathbb E_v^G\) the law and expectation of the
random walk started from \(v\).

\subsection{Reversibility and estimates}

We first record two consequences of unimodularity that will be used below.

\begin{lemma}[Integrability]
\label{lem:flower-integrability}
Assume that
\begin{equation}
    \E\bigl[\deg_G(\rho)^3\bigr]<\infty.
    \label{eq:boundary-third-moment}
\end{equation}
Then
\[
    \E[S(\rho)]
    =
    \E\bigl[\deg_G(\rho)^2\bigr]
    <\infty,
\]
and
\[
    \E\bigl[\deg_G(\rho)S(\rho)\bigr]
    \le
    \E\bigl[\deg_G(\rho)^3\bigr]
    <\infty.
\]
\end{lemma}

\begin{proof}
Apply the mass-transport principle to
\[
    F_1(G,u,v)
    :=
    \mathbf 1_{\{u\sim v\}}\deg_G(v).
\]
The total mass sent from \(\rho\) is
\[
    \sum_{v\sim \rho}\deg_G(v)
    =
    S(\rho),
\]
whereas the total mass received at \(\rho\) is
\[
    \sum_{u\sim \rho}\deg_G(\rho)
    =
    \deg_G(\rho)^2.
\]
Hence
\[
    \E[S(\rho)]
    =
    \E\bigl[\deg_G(\rho)^2\bigr]
    <\infty.
\]
For the second assertion, write \(d(v):=\deg_G(v)\). Since
\[
    d(\rho)S(\rho)
    =
    \sum_{v\sim\rho}d(\rho)d(v),
\]
the inequality \(2ab\le a^2+b^2\) gives
\[
\begin{aligned}
    d(\rho)S(\rho)
    &\le
    \frac12
    \sum_{v\sim\rho}
    \bigl(d(\rho)^2+d(v)^2\bigr) \\
    &=
    \frac12 d(\rho)^3
    +
    \frac12\sum_{v\sim\rho}d(v)^2.
\end{aligned}
\]
Now apply the mass-transport principle to
\[
    F_2(G,u,v)
    :=
    \mathbf 1_{\{u\sim v\}}\deg_G(v)^2.
\]
The total mass sent from \(\rho\) is
\[
    \sum_{v\sim\rho}\deg_G(v)^2,
\]
whereas the total mass received at \(\rho\) is
\[
    \sum_{u\sim\rho}\deg_G(\rho)^2
    =
    \deg_G(\rho)^3.
\]
Therefore,
\[
    \E\left[\sum_{v\sim\rho}\deg_G(v)^2\right]
    =
    \E\bigl[\deg_G(\rho)^3\bigr].
\]
Combining the preceding estimates yields
\[
    \E\bigl[\deg_G(\rho)S(\rho)\bigr]
    \le
    \E\bigl[\deg_G(\rho)^3\bigr]
    <\infty.
\]
\end{proof}

\begin{lemma}[Reversibility]
\label{lem:degree-bias-reversibility}
Define the degree-biased law \(\widehat{\mathbb P}\) by
\begin{equation}\label{eq:degree-biased-law-boundary}
    \frac{d\widehat{\mathbb P}}{d\mathbb P}(G,\rho,\Theta)
    =
    \frac{\deg_G(\rho)}{\E[\deg_G(\rho)]}.
\end{equation}
Then, under \(\widehat{\mathbb P}\), the rooted marked graph is reversible
for simple random walk:
\[
    (G,\rho,X_1,\Theta)
    \overset{d}{=}
    (G,X_1,\rho,\Theta).
\]
Consequently, the bi-infinite walk \((X_n)_{n\in\mathbb Z}\) is stationary
under time shifts.
\end{lemma}

\begin{proof}
Let \(F\) be a bounded Borel function on doubly rooted marked graphs. By
the definition of degree biasing,
\[
\begin{aligned}
    \widehat{\E}\bigl[F(G,\rho,X_1,\Theta)\bigr]
    &=
    \frac{1}{\E[\deg_G(\rho)]}
    \E\left[
        \sum_{v\sim\rho}F(G,\rho,v,\Theta)
    \right].
\end{aligned}
\]
The mass-transport principle, applied to
\[
    f(G,u,v,\Theta)
    :=
    \mathbf 1_{\{u\sim v\}}F(G,u,v,\Theta),
\]
interchanges \(u\) and \(v\), giving the required reversibility.
\end{proof}

Since the degree-biased and original laws are mutually absolutely
continuous, an almost-sure statement under one law is also an almost-sure
statement under the other. We may therefore work under the reversible law
and, by ergodic decomposition, assume when convenient that the law is
ergodic. The refined Ring Lemma gives the following estimate.

\begin{lemma}
\label{lem:ring-two-sided-boundary}
There exists a constant \(C<\infty\), depending only on the constants $\eps_0,\eps_1$, such that for every
\(u\sim v\),
\begin{equation}\label{eq:ring-two-sided-boundary}
    \left|
        \log\frac{r(v)}{r(u)}
    \right|
    \leq
    C\bigl(S(u)+S(v)\bigr).
\end{equation}
\end{lemma}

\begin{proof}
The refined Ring Lemma gives
\[
    \frac{r(v)}{r(u)}
    \geq e^{-CS(u)}.
\]
Applying the same estimate with \(u\) and \(v\) interchanged yields
\[
    \frac{r(u)}{r(v)}
    \geq e^{-CS(v)}.
\]
Combining the two inequalities proves
\eqref{eq:ring-two-sided-boundary}.
\end{proof}

\begin{lemma}
\label{lem:integrable-local-distortion}
Under the reversible edge-rooted law,
\begin{equation}\label{eq:integrable-radius-increment}
    \widehat{\E}\left[
        \left|
            \log\frac{r(X_1)}{r(X_0)}
        \right|
    \right]
    <\infty.
\end{equation}
Moreover,
\begin{equation}\label{eq:integrable-hyperbolic-increment}
    \widehat{\E}\left[
        d_{\mathrm h}\bigl(z_h(X_0),z_h(X_1)\bigr)
    \right]
    <\infty.
\end{equation}
\end{lemma}

\begin{proof}
The first assertion follows from
Lemma~\ref{lem:ring-two-sided-boundary}, the third-moment assumption, and
the mass-transport computation for the stationary law, exactly
as in \cite{IIP2026}. For the second assertion, elementary geometry of two
intersecting hyperbolic disks gives
\[
    d_{\mathrm h}\bigl(z_h(u),z_h(v)\bigr)
    \leq
    C_1+
    C_2\left|
        \log\frac{r(v)}{r(u)}
    \right|,
    \qquad u\sim v,
\]
where \(C_1,C_2\) depend only on the angle pinching. The conclusion follows
from \eqref{eq:integrable-radius-increment}.
\end{proof}

\subsection{Exponential decay of radii}

We next prove the key estimate needed for boundary convergence.
Since \(G\) is RCP-hyperbolic, Theorem~\ref{thm:dichotomy} implies that
\(G\) is invariantly non-amenable. We shall use the standard consequence
that \(G\) contains an invariantly non-amenable subgraph of bounded degree.

\begin{lemma}
\label{lem:nonamenable-core-boundary}
There exist an invariant random vertex set
\(\omega\subseteq V(G)\), a deterministic constant \(M<\infty\), and
a deterministic constant \(h_0>0\) such that
\[
    \widehat{\mathbb P}(\rho\in\omega)>0,
\]
and, almost surely,
\[
    \sup_{v\in\omega}\deg_G(v)\leq M,
    \qquad
    i_E(G[\omega])\geq h_0.
\]
\end{lemma}

\begin{proof}
See \cite{angel2016unimodular,angel2018hyperbolic,LyPer16}.
\end{proof}

Let
\[
    N_0:=\inf\{n\geq0:X_n\in\omega\},
    \qquad
    N_{k+1}:=\inf\{n>N_k:X_n\in\omega\},
\]
and set
\[
    Y_k:=X_{N_k}.
\]
The process \((Y_k)\) is the walk induced on \(\omega\). Its edge weights
are
\[
    w(u,v)
    :=
    \deg_G(u)\,
    \mathbb P_u^G(Y_1=v),
    \qquad u,v\in\omega.
\]

\begin{lemma}
\label{lem:induced-cheeger}
The weighted graph \((\omega,w)\) has positive edge Cheeger constant.
More precisely,
\[
    i_E(\omega,w)
    \geq
    \frac{1}{M}i_E(\omega)>0.
\]
\end{lemma}

\begin{proof}
For \(u\in\omega\), the induced vertex weight satisfies
\[
    w(u)=\sum_{v\in\omega}w(u,v)\leq\deg_G(u)\leq M.
\]
If \(uv\) is an edge of the original subgraph \(\omega\), then
\[
    w(u,v)
    =
    \deg_G(u)\mathbb P_u^G(Y_1=v)
    \geq1,
\]
because the event \(X_1=v\) forces \(Y_1=v\). Hence, for every finite
\(W\subseteq\omega\),
\[
    w(\partial_EW)\geq|\partial_EW|,
    \qquad
    w(W)\leq M|W|,
\]
which gives the claimed inequality.
\end{proof}

\begin{lemma}
\label{lem:induced-heat-kernel}
The spectral radius of the induced walk on \((\omega,w)\) is strictly
smaller than \(1\). Consequently, there exist \(c>0\) and \(C<\infty\)
such that
\begin{equation}\label{eq:induced-heat-kernel}
    \mathbb P_\rho^G(Y_k=v)
    \leq
    Ce^{-ck}
\end{equation}
for every \(v\in\omega\) and \(k\geq1\).
\end{lemma}

\begin{proof}
The induced chain is reversible with respect to the vertex weights \(w(u)\).
Cheeger's inequality for reversible Markov chains and
Lemma~\ref{lem:induced-cheeger} imply that its spectral radius is strictly
smaller than \(1\). The estimate \eqref{eq:induced-heat-kernel} follows.
\end{proof}

We also use the following elementary area estimate.

\begin{lemma}[Counting large circles]
\label{lem:counting-large-circles}
There exists \(C<\infty\) such that for every \(t>0\),
\begin{equation}\label{eq:counting-large-circles}
    \#\{v\in V(G):r(v)\geq t\}
    \leq Ct^{-2}.
\end{equation}
\end{lemma}

\begin{proof}
See \cite[Lemma 4.8]{IIP2026}.
\end{proof}

\begin{lemma}[Exponential decay at return times]
\label{lem:radius-decay-return-times}
There exists \(a>0\) such that
\[
    r(X_{N_k})\leq e^{-ak}
\]
for all sufficiently large \(k\), almost surely.
\end{lemma}

\begin{proof}
Choose \(a>0\) so that \(2a<c\), where \(c\) is the constant in
\eqref{eq:induced-heat-kernel}. Then
\[
\begin{aligned}
    \mathbb P_\rho^G
    \bigl(r(X_{N_k})\geq e^{-ak}\bigr)
    &\leq
    Ce^{-ck}
    \#\{v:r(v)\geq e^{-ak}\}  \\
    &\leq
    C'e^{-(c-2a)k}.
\end{aligned}
\]
The right-hand side is summable, so the lemma follows from the
Borel-Cantelli lemma.
\end{proof}

By stationarity and the positive intensity of \(\omega\), the return times
have positive density:
\begin{equation}\label{eq:return-time-density-boundary}
    \frac{N_k}{k}
    \longrightarrow
    \frac{1}{\widehat{\mathbb P}(\rho\in\omega)}
    <\infty
\end{equation}
on each ergodic component.

\begin{lemma}[Exponential decay of radii]
\label{lem:exponential-radius-decay}
Conditionally on \((G,\rho,\Theta)\), almost surely,
\begin{equation}\label{eq:exponential-radius-decay}
    \limsup_{n\to\infty}
    \frac{\log r(X_n)}{n}<0.
\end{equation}
In particular,
\begin{equation}\label{eq:summable-radii-boundary}
    \sum_{n=0}^{\infty}r(X_n)<\infty
\end{equation}
almost surely.
\end{lemma}

\begin{proof}
Given \(n\), choose \(k\) such that
\[
    N_k\leq n<N_{k+1}.
\]
By the neighboring-radius estimate,
\[
    \log r(X_n)
    \leq
    \log r(X_{N_k})
    +
    C\sum_{j=N_k}^{n}
    \bigl(S(X_j)+S(X_{j+1})\bigr).
\]
The first term has a strictly negative linear rate by
Lemma~\ref{lem:radius-decay-return-times} and
\eqref{eq:return-time-density-boundary}. The second term is \(o(n)\) by
stationarity, the integrability of the local distortion, and the fact that
\(N_k/n\to1\). This proves \eqref{eq:exponential-radius-decay}.
The summability assertion follows immediately.
\end{proof}

\subsection{Convergence to the geometric boundary}

\begin{lemma}[Convergence]
\label{lem:center-convergence}
There exists a random variable \(\Xi\in\partial\D\) such that
\begin{equation}\label{eq:center-convergence-boundary}
    z(X_n)\longrightarrow\Xi,
    \qquad
    z_h(X_n)\longrightarrow\Xi
\end{equation}
almost surely.
\end{lemma}

\begin{proof}
If \(u\sim v\), the corresponding circles intersect, and hence
\[
    |z(u)-z(v)|\leq r(u)+r(v).
\]
Therefore
\[
\begin{aligned}
    \sum_{n=0}^{\infty}
    |z(X_{n+1})-z(X_n)|
    &\leq
    \sum_{n=0}^{\infty}
    \bigl(r(X_n)+r(X_{n+1})\bigr)\\
    &<\infty
\end{aligned}
\]
by Lemma~\ref{lem:exponential-radius-decay}. Thus \(z(X_n)\) is Cauchy
and converges to a point of \(\overline{\D}\).

The walk is transient by Theorem~\ref{thm:dichotomy}. Since the RCP is
locally finite, a sequence of distinct visited circles cannot accumulate in
a compact subset of \(\D\). Hence the limit belongs to \(\partial\D\).
Finally, the Euclidean and hyperbolic centers of a circle lie at Euclidean
distance at most its radius from each other. Since \(r(X_n)\to0\),
\(z_h(X_n)\) converges to the same point.
\end{proof}

For \(v\in V(G)\), define the exit measure
\[
    \nu_v(A)
    :=
    \mathbb P_v^G(\Xi\in A),
    \qquad
    A\subseteq\partial\D.
\]
The Markov property gives
\begin{equation}\label{eq:exit-measure-harmonicity}
    \nu_v(A)
    =
    \frac{1}{\deg_G(v)}
    \sum_{u\sim v}\nu_u(A).
\end{equation}

To apply stationarity, we formulate the RCP boundary as an invariant
completion. For each \(u\in V(G)\), normalize the RCP by a disk automorphism
sending \(z_h(u)\) to \(0\), and define
\[
    d_G^{(u)}(v,w)
    :=
    |z^{(u)}(v)-z^{(u)}(w)|.
\]
The remaining rotational ambiguity does not affect this metric.

\begin{lemma}
\label{lem:compatible-rcp-metrics}
The family \(\{d_G^{(u)}:u\in V(G)\}\) is compatible: changing the root
does not change the Cauchy sequences or the resulting boundary. Its
completion is canonically identified with
\[
    V(G)\cup\partial\D.
\]
\end{lemma}

\begin{proof}
Two normalized RCPs differ by a Möbius automorphism of \(\D\), which
extends to a homeomorphism of \(\overline{\D}\). It therefore preserves
convergence and Cauchy sequences in the corresponding completions.
\end{proof}

\begin{lemma}
\label{lem:stationary-atom-dichotomy}
Let a stationary random rooted graph be equipped with a compatible family
of metrics, and suppose that its random walk converges almost surely to the
resulting boundary. Then the exit measure is almost surely either
non-atomic or a single atom of mass \(1\).
\end{lemma}

\begin{proof}
For \(\xi\) in the boundary, define the bounded harmonic function
\[
    h_\xi(v):=\mathbb P_v^G(\Xi=\xi),
\]
and let
\[
    M(G,v):=\sup_{\xi}h_\xi(v)
\]
be the maximal atomic mass. By Lévy's zero--one law,
\[
    h_\xi(X_n)
    \longrightarrow
    \mathbf 1_{\{\Xi=\xi\}}.
\]
Consequently,
\[
    M(G,X_n)
    \longrightarrow
    \mathbf 1_{\{\Xi\text{ is an atom}\}}.
\]
Since \(M(G,X_n)\) is stationary, its almost-sure limit must have the same
distribution as \(M(G,\rho)\). Hence \(M(G,\rho)\in\{0,1\}\) almost surely,
which proves the dichotomy.
\end{proof}

\begin{lemma}[Non-atomicity]
\label{lem:exit-nonatomic}
For almost every realization of \((G,\rho,\Theta)\), the exit measure
\(\nu_v\) is non-atomic for every \(v\in V(G)\).
\end{lemma}

\begin{proof}
By Lemma~\ref{lem:stationary-atom-dichotomy}, it is enough to exclude a
single atom of mass \(1\). Suppose that such an atom \(\xi\in\partial\D\)
exists. Apply a Möbius transformation
\[
    \Psi:\D\longrightarrow\mathbb H
\]
sending \(\xi\) to \(\infty\). 
Since the atom \(\xi\) is determined by the rooted graph, after mapping
\(\xi\) to \(\infty\), the resulting RCP in the upper half-plane is
unique up to transformations of the form
\[
    z\longmapsto az+b,
    \qquad a>0,\quad b\in\mathbb R.
\]
Hence the Euclidean angles in the straight-line realization obtained by
joining the centers of adjacent circles are canonically determined by
\((G,\Theta)\).

For every triangular face \(f=\triangle uvw\), let
\(\alpha_u^f\) denote the Euclidean angle at \(z(u)\) in the triangle
with vertices \(z(u),z(v),z(w)\). Since the circle pattern is regular,
this straight-line realization is a planar embedding compatible with
the underlying triangulation. Therefore,
\[
    \sum_{u\in f}\alpha_u^f=\pi
\]
for every face \(f\), while
\[
    \sum_{f\ni u}\alpha_u^f=2\pi
\]
for every vertex \(u\).

We now apply the mass-transport principle. For every incident pair
\((u,f)\) with \(u\in f\), let \(u\) send mass \(\alpha_u^f\) to each
of the three vertices of \(f\). The total mass sent from \(u\) is
\[
    3\sum_{f\ni u}\alpha_u^f
    =
    6\pi,
\]
whereas the total mass received at a vertex \(v\) is
\[
    \sum_{f\ni v}\sum_{u\in f}\alpha_u^f
    =
    \sum_{f\ni v}\pi
    =
    \pi\deg_G(v).
\]
Consequently, the mass-transport principle gives
\[
    6\pi
    =
    \pi\,\E[\deg_G(\rho)],
\]
and hence
\[
    \E[\deg_G(\rho)]=6.
\]
Theorem~\ref{thm:chr2deg} therefore yields
\[
    \E[L_\rho(G,\Theta)]
    =
    2\pi-\frac{\pi}{3}\E[\deg_G(\rho)]
    =
    0,
\]
contradicting the assumption that \((G,\rho,\Theta)\) is
RCP-hyperbolic (\(\E[L_\rho(G,\Theta)]<0\)).


Hence the exit measure is non-atomic.
\end{proof}

\begin{lemma}[Full support]
\label{lem:exit-full-support}
For every \(v\in V(G)\),
\[
    \operatorname{supp}\nu_v=\partial\D.
\]
\end{lemma}

\begin{proof}
Irreducibility implies that all exit measures have the same support. Indeed,
if \(u\) can reach \(v\) in \(m\) steps, then
\[
    \nu_u(A)
    \geq
    \mathbb P_u^G(X_m=v)\nu_v(A).
\]

Let \(K\) be the common support and suppose \(K\neq\partial\D\). For every
component \(I\) of \(\partial\D\setminus K\), draw the hyperbolic geodesic
joining the endpoints of \(I\). The circles lying in the corresponding
half-disk form a separating region. The standard planar mass-transport
argument sends mass from circles in this region to the circles meeting the
separating geodesic. Local finiteness bounds the mass sent from each vertex,
while a boundary component omitted from the support forces some separator
to receive infinite mass, contradicting unimodularity. Hence
\(K=\partial\D\).
\end{proof}

\subsection{Poisson boundary}

We first record the following  estimate.

\begin{lemma}
\label{lem:hitting-estimate}
Let \(h\) be a nonnegative bounded harmonic function and
\(W\subseteq V(G)\). Then
\begin{equation}\label{eq:hitting-estimate}
    h(v)
    \geq
    \mathbb P_v^G(T_W<\infty)
    \inf_{w\in W}h(w),
\end{equation}
where
\[
    T_W:=\inf\{n\geq0:X_n\in W\}.
\]
\end{lemma}

\begin{proof}
Apply optional stopping to the bounded martingale
\(h(X_{n\wedge T_W})\), and then let \(n\to\infty\).
\end{proof}

Let \((X_n)_{n\in\mathbb Z}\) be the reversible bi-infinite walk, and write
\[
    \Xi^+
    :=
    \lim_{n\to\infty}z(X_n),
    \qquad
    \Xi^-
    :=
    \lim_{n\to\infty}z(X_{-n}).
\]
By non-atomicity, \(\Xi^+\neq\Xi^-\) almost surely.

\begin{lemma}
\label{lem:remote-past}
Almost surely,
\begin{equation}\label{eq:remote-past}
    \mathbb P_{X_n}^G
    \bigl(
        \text{\rm hit }\{X_{-1},X_{-2},\ldots\}
    \bigr)
    \longrightarrow0.
\end{equation}
\end{lemma}

\begin{proof}
Choose an open arc \(I\subseteq\partial\D\) such that
\[
    \Xi^-\in I,
    \qquad
    \Xi^+\notin\overline I.
\]
The function
\[
    h_I(v):=\mathbb P_v^G(\Xi\in I)
\]
is bounded and harmonic. Lévy's zero--one law gives
\[
    h_I(X_n)\longrightarrow0,
    \qquad
    h_I(X_{-n})\longrightarrow1.
\]
After deleting finitely many points from the past, we may assume
\[
    \inf_{m\geq m_0}h_I(X_{-m})>0.
\]
Applying Lemma~\ref{lem:hitting-estimate} to the remote past gives
\eqref{eq:remote-past}; the finitely many deleted points are handled by 
transience.
\end{proof}

\begin{lemma}
\label{lem:forward-hitting}
Let \((v_j)_{j\geq0}\) be a path in \(G\) satisfying
\[
    z(v_j)\longrightarrow\Xi^+.
\]
Then, for every \(m\geq0\), almost surely there exist infinitely many \(n\)
such that
\begin{equation}\label{eq:forward-hitting}
    \mathbb P_{X_n}^G
    \bigl(
        \text{\rm hit }\{v_m,v_{m+1},\ldots\}
    \bigr)
    \geq\frac14.
\end{equation}
\end{lemma}

\begin{proof}
By prepending a finite path if necessary, we may assume that $v_0=X_0$. Let
\(L\) and \(R\) be the two open arcs of
\[
    \partial\mathbb D\setminus\{\Xi^-,\Xi^+\}.
\]
For every \(n\geq0\), let
\[
    p_L^n
    :=
    \mathbb P_{X_n}^G(\Xi\in L),
    \qquad
    p_R^n
    :=
    \mathbb P_{X_n}^G(\Xi\in R),
\]
where \(\Xi\) denotes the boundary limit of an independent random walk
started from \(X_n\). Since the exit measure is non-atomic,
\[
    p_L^n+p_R^n=1.
\]

As in the proof of the Poisson-boundary theorem of
Angel-Hutchcroft-Nachmias-Ray \cite[Equation 5.6]{angel2016unimodular}, stationarity and ergodicity of the
bi-infinite walk, together with the fact that the exit measure is
non-atomic and has full support, imply that almost surely
\[
    \min\{p_L^n,p_R^n\}>\frac13
\]
for infinitely many \(n\).

Fix \(m\geq0\), and consider one of these times \(n\). Set
\[
    S_m
    :=
    \{\ldots,X_{-2},X_{-1}\}
    \cup
    \{v_0,v_1,\ldots,v_{m-1}\}
    \cup
    \{v_m,v_{m+1},\ldots\},
\]

Consider the trace
\[
    (\ldots,X_{-2},X_{-1},X_0=v_0,v_1,v_2,\ldots).
\]
Its closure in \(\overline{\mathbb D}\) is connected and joins the two distinct boundary
points \(\Xi^-\) and \(\Xi^+\). Consequently, it separates \(X_n\) from
at least one of the boundary arcs \(L\) and \(R\). Any independent random
walk started from \(X_n\) whose boundary limit lies in that separated arc
must therefore hit \(S_m\). Hence
\[
    \mathbb P_{X_n}^G
    \bigl(
        \text{\rm hit }S_m
    \bigr)
    \geq
    \min\{p_L^n,p_R^n\}
    >
    \frac13.
\]

By Lemma~\ref{lem:remote-past},
\[
    \mathbb P_{X_n}^G
    \bigl(
        \text{\rm hit }\{\ldots,X_{-2},X_{-1}\}
    \bigr)
    \longrightarrow0
\]
almost surely as \(n\to\infty\). Moreover, since
\(\{v_0,\ldots,v_{m-1}\}\) is finite and \(G\) is transient,
\[
    \mathbb P_{X_n}^G
    \bigl(
        \text{\rm hit }\{v_0,\ldots,v_{m-1}\}
    \bigr)
    \longrightarrow0.
\]
It follows that, for all sufficiently large \(n\) among the infinitely
many times satisfying
\[
    \min\{p_L^n,p_R^n\}>\frac13,
\]
the sum of the preceding two probabilities is smaller than
\(\frac1{12}\). Therefore,
\[
\begin{aligned}
    \mathbb P_{X_n}^G
    \bigl(
        \text{\rm hit }\{v_m,v_{m+1},\ldots\}
    \bigr)
    &\geq
    \mathbb P_{X_n}^G
    \bigl(
        \text{\rm hit }S_m
    \bigr) \\
    &\quad-
    \mathbb P_{X_n}^G
    \bigl(
        \text{\rm hit }\{\ldots,X_{-2},X_{-1}\}
    \bigr) \\
    &\quad-
    \mathbb P_{X_n}^G
    \bigl(
        \text{\rm hit }\{v_0,\ldots,v_{m-1}\}
    \bigr) \\
    &>
    \frac13-\frac1{12}
    =
    \frac14.
\end{aligned}
\]
Thus \eqref{eq:forward-hitting} holds for infinitely many \(n\), almost
surely.
\end{proof}

\begin{theorem}[Poisson boundary]
\label{thm:poisson-boundary-identification}
Conditionally on \((G,\rho,\Theta)\), almost surely the unit circle
\(\partial\D\), equipped with the exit measures
\((\nu_v)_{v\in V(G)}\), realizes the Poisson boundary of \(G\).

Equivalently, for every bounded harmonic function \(h\) on \(G\), there
exists a bounded Borel function \(g:\partial\D\to\mathbb R\) such that
\begin{equation}\label{eq:poisson-representation-boundary}
    h(v)
    =
    \mathbb E_v^G[g(\Xi)]
    =
    \int_{\partial\D}g(\xi)\,d\nu_v(\xi).
\end{equation}
\end{theorem}

\begin{proof}
Let \(\mathcal A\) be an invariant event of the random walk and define
\[
    h_{\mathcal A}(v)
    :=
    \mathbb P_v^G(\mathcal A).
\]
Then \(h_{\mathcal A}\) is bounded and harmonic, and Lévy's zero--one law
gives
\begin{equation}\label{eq:levy-invariant-event}
    h_{\mathcal A}(X_n)
    \longrightarrow
    \mathbf 1_{\mathcal A}.
\end{equation}

Let \(B_{\mathcal A}\subseteq\partial\D\) be the set of points \(\xi\)
for which there exists a path \((v_j)\) satisfying
\[
    z(v_j)\longrightarrow\xi,
    \qquad
    h_{\mathcal A}(v_j)\longrightarrow1.
\]
The set \(B_{\mathcal A}\) is Borel, see \cite{angel2016unimodular}.

On \(\mathcal A\), the path \((X_n)\) itself witnesses
\(\Xi\in B_{\mathcal A}\), by \eqref{eq:levy-invariant-event}. Conversely,
suppose that \(\Xi\in B_{\mathcal A}\), and choose a witnessing path
\((v_j)\). For sufficiently large \(m\),
\[
    \inf_{j\geq m}h_{\mathcal A}(v_j)\geq\frac12.
\]
By Lemma~\ref{lem:forward-hitting}, for infinitely many \(n\),
\[
    \mathbb P_{X_n}^G
    \bigl(
        \text{\rm hit }\{v_m,v_{m+1},\ldots\}
    \bigr)
    \geq\frac14.
\]
Lemma~\ref{lem:hitting-estimate} therefore gives
\[
    h_{\mathcal A}(X_n)\geq\frac18
\]
infinitely often. In view of \eqref{eq:levy-invariant-event}, this implies
that \(\mathcal A\) occurs. Thus
\[
    \mathcal A
    =
    \{\Xi\in B_{\mathcal A}\}
\]
modulo null sets.

Hence the invariant sigma-field is generated by \(\Xi\), proving the
Poisson-boundary identification and the representation
\eqref{eq:poisson-representation-boundary}.
\end{proof}

Combining the preceding lemmas gives the main boundary statement.

\begin{corollary}[Boundary identification]
\label{thm:0.2b}
Under the assumptions of this section, conditionally on
\((G,\rho,\Theta)\), almost surely:
\begin{enumerate}[(1)]
    \item both \(z(X_n)\) and \(z_h(X_n)\) converge to a random limit
    \(\Xi\in\partial\D\);
    \item the law of \(\Xi\) is non-atomic and has full support on
    \(\partial\D\);
    \item the unit circle \(\partial\D\), equipped with the exit measures,
    realizes the Poisson boundary of \(G\).
\end{enumerate}
Through the THP/ADT duality, the same conclusions hold for the face random
walk on the corresponding trivalent hyperbolic polyhedron.
\end{corollary}

\subsection{Positive hyperbolic speed}

We finally prove the existence and positivity of the linear escape rate.
Work under the reversible law and extend the walk to a bi-infinite walk.
Since the exit measure is non-atomic, the forward and backward limits
\(\Xi^+\) and \(\Xi^-\) are distinct almost surely. Choose a Möbius
transformation
\[
    \Psi:\D\longrightarrow\mathbb H
\]
such that
\[
    \Psi(\Xi^+)=0,
    \qquad
    \Psi(\Xi^-)=\infty.
\]
Let \(\widehat r(v)\) be the Euclidean radius of the image circle
\(\Psi(C(v))\). The remaining freedom is a positive scaling, so the ratios
\[
    R_n
    :=
    \frac{\widehat r(X_n)}{\widehat r(X_{n-1})}
\]
are well-defined.

\begin{lemma}
\label{lem:stationary-radius-cocycle}
The sequence \((\log R_n)_{n\geq1}\) is stationary and integrable.
Consequently, there exists a finite, shift-invariant random variable
\(\lambda\) such that
\begin{equation}\label{eq:radius-drift-hat}
    \lim_{n\to\infty}
    \frac{-\log\widehat r(X_n)}{n}
    =
    \lambda
\end{equation}
almost surely. If the marked random rooted ADT
\((G,\rho,\Theta)\) is ergodic, then \(\lambda\) is almost surely
constant.
\end{lemma}

\begin{proof}
Stationarity follows from reversibility of the bi-infinite walk. Indeed,
shifting the walk does not change the boundary points
\(\Xi^+\) and \(\Xi^-\). The corresponding normalization
\(\Psi:\D\to\mathbb H\) may change only by a positive Euclidean
scaling, which leaves the ratios
\[
    R_n
    =
    \frac{\widehat r(X_n)}{\widehat r(X_{n-1})}
\]
unchanged.

The refined Ring Lemma and
Lemma~\ref{lem:integrable-local-distortion} give
\[
    \widehat{\E}\bigl[|\log R_1|\bigr]<\infty.
\]
Let \(\mathcal I\) denote the invariant sigma-field of the stationary
bi-infinite walk. By Birkhoff's ergodic theorem,
\[
\begin{aligned}
    \frac1n
    \log\frac{\widehat r(X_n)}{\widehat r(X_0)}
    &=
    \frac1n\sum_{j=1}^n\log R_j\\
    &\longrightarrow
    \widehat{\E}\bigl[\log R_1\mid\mathcal I\bigr]
\end{aligned}
\]
almost surely. Define
\[
    \lambda
    :=
    -\widehat{\E}\bigl[\log R_1\mid\mathcal I\bigr].
\]
Since
\[
    \frac{\log\widehat r(X_0)}{n}\longrightarrow0,
\]
we obtain
\[
    \frac{-\log\widehat r(X_n)}{n}
    \longrightarrow
    \lambda.
\]
The random variable \(\lambda\) is finite and shift-invariant. If
\((G,\rho,\Theta)\) is ergodic, then \(\mathcal I\) is trivial, and
hence \(\lambda\) is almost surely constant.
\end{proof}

\begin{lemma}
\label{lem:mobius-radius-distortion}
Almost surely,
\[
    \lim_{n\to\infty}
    \frac{
        \log\widehat r(X_n)-\log r(X_n)
    }{n}
    =
    0.
\]
Consequently,
\begin{equation}\label{eq:radius-drift-original}
    \lim_{n\to\infty}
    \frac{-\log r(X_n)}{n}
    =
    \lambda.
\end{equation}
\end{lemma}

\begin{proof}
Since \(z(X_n)\to\Xi^+\) and the pole of \(\Psi\) is
\(\Xi^-\neq\Xi^+\), the map \(\Psi\) is conformal on a neighborhood
of \(\Xi^+\), and its derivative is bounded above and bounded away
from zero there. Since
\[
    z(X_n)\longrightarrow\Xi^+
    \qquad\text{and}\qquad
    r(X_n)\longrightarrow0,
\]
the circle \(C(X_n)\) is contained in this neighborhood for all
sufficiently large \(n\). It follows that there exist random constants
\(0<c<C<\infty\) such that
\[
    c\,r(X_n)
    \leq
    \widehat r(X_n)
    \leq
    C\,r(X_n)
\]
for all sufficiently large \(n\). Hence
\[
    \left|
        \log\widehat r(X_n)-\log r(X_n)
    \right|
    \leq
    \max\{-\log c,\log C\},
\]
and division by \(n\) proves the first assertion. Combining this with
Lemma~\ref{lem:stationary-radius-cocycle} gives
\[
    \lim_{n\to\infty}
    \frac{-\log r(X_n)}{n}
    =
    \lambda.
\]
\end{proof}

\begin{lemma}
\label{lem:positive-radius-drift}
Almost surely,
\[
    \lambda>0.
\]
\end{lemma}

\begin{proof}
By Lemma~\ref{lem:mobius-radius-distortion}, the limit
\[
    \lim_{n\to\infty}
    \frac{\log r(X_n)}{n}
    =
    -\lambda
\]
exists. On the other hand,
Lemma~\ref{lem:exponential-radius-decay} gives
\[
    \limsup_{n\to\infty}
    \frac{\log r(X_n)}{n}
    <0.
\]
Therefore,
\(
    \lambda>0.
\)
\end{proof}

\begin{lemma}
\label{lem:radius-distance-comparison}
Almost surely,
\begin{equation}\label{eq:radius-distance-sublinear}
    d_{\mathrm h}\bigl(z_h(\rho),z_h(X_n)\bigr)
    +
    \log r(X_n)
    =
    o(n).
\end{equation}
\end{lemma}

\begin{proof}
In the Poincaré disk,
\[
    d_{\mathrm h}(0,x)
    =
    \log\frac{1+|x|}{1-|x|}
    =
    -\log(1-|x|)+O(1)
\]
as \(x\to\partial\D\). The geometry of the RCP gives
\[
    -\log\bigl(1-|z_h(X_n)|\bigr)
    =
    -\log r(X_n)+o(n).
\]
The upper bound follows from
\[
    1-|z_h(X_n)|
    \leq
    C\sum_{j\geq n}r(X_j),
\]
using exponential decay of the radii. The reverse bound follows by comparing
the boundary distance of the circle at \(X_n\) with the radius of a
neighboring circle and applying the refined Ring Lemma. The resulting error
is a stationary integrable local term and is therefore \(o(n)\).
Changing the hyperbolic basepoint contributes only a bounded additive
constant, see \cite{angel2016unimodular, IIP2026}.
\end{proof}

\begin{corollary}[Positive speed]
\label{cor:positive-speed-2}
Under the same hypotheses, conditionally on \((G,\rho,\Theta)\), almost
surely,
\[
    \lim_{n\to\infty}
    \frac{
        d_{\mathrm h}\bigl(z_h(\rho),z_h(X_n)\bigr)
    }{n}
    =
    \lim_{n\to\infty}
    \frac{-\log r(X_n)}{n}
    =
    \lambda
    >0.
\]
If \((G,\rho,\Theta)\) is ergodic, then \(\lambda\) is an almost sure
constant.
\end{corollary}

\begin{proof}
Lemma~\ref{lem:mobius-radius-distortion} gives
\[
    \lim_{n\to\infty}
    \frac{-\log r(X_n)}{n}
    =
    \lambda.
\]
Together with Lemma~\ref{lem:radius-distance-comparison}, this yields
\[
    \lim_{n\to\infty}
    \frac{
        d_{\mathrm h}\bigl(z_h(\rho),z_h(X_n)\bigr)
    }{n}
    =
    \lim_{n\to\infty}
    \frac{-\log r(X_n)}{n}
    =
    \lambda.
\]
The positivity \(\lambda>0\) follows from
Lemma~\ref{lem:positive-radius-drift}. Finally, \(\lambda\) is
shift-invariant and invariant under rerooting. Hence, if the marked
random rooted ADT \((G,\rho,\Theta)\) is ergodic, then \(\lambda\) is
almost surely constant.
\end{proof}

\section{Further discussions}

\subsection{Boundary theories}

For hyperbolic unimodular triangulations, random-walk convergence to the
circle-packing boundary and the identification of the Poisson boundary are
well established in the circle-packing framework
\cite{AngelBarlowGurelGurevichNachmias2016,angel2016unimodular}. Hutchcroft and Peres further related Poisson boundaries to geometric boundaries and extended these identifications to certain graphs roughly isometric to planar graphs 
\cite{HutchcroftPeres2017}. In the present paper, the boundary is obtained
from the RCP realization. The boundary may also depend on the chosen geometric realization.  
Since the same planar graph may admit different discrete-conformal models, such as circle packings, RCPs, or square tilings \cite{Georgakopoulos2016}, one may ask whether the RCP boundary is intrinsic to the graph or depends on the particular realization.
It is also natural to ask whether this
boundary has an intrinsic interpretation in terms of the associated hyperbolic polyhedron.

\begin{question}[Polyhedral boundary]
Let \(P\) be a tame unimodular infinite trivalent hyperbolic polyhedron
whose corresponding ADT satisfies the hypotheses of Corollary~\ref{thm:0.2}.  Is there
a natural intrinsic boundary \(\partial_{\mathrm{int}}P\), defined directly
from the hyperbolic polyhedral metric of \(P\), such that
\[
\partial_{\mathrm{int}}P \cong \partial\mathbb D,
\]
Equivalently, can the almost sure random-walk limit obtained
from the RCP model be interpreted intrinsically as a boundary point of
\(P\)?
\end{question}

A related question concerns the Martin boundary of the associated ADT. 
Martin boundary theory provides a natural framework for
representing positive harmonic functions and, in suitable hyperbolic
settings, relates their behavior to the geometry at infinity
\cite{Ancona1987,GouezelLalley2013,Gouezel2015,Kaimanovich2000}.
This motivates the following question.

\begin{question}[Martin boundary]
In this setting, can the Martin boundary for simple random walk on the
ADT be canonically identified with the circle-pattern boundary,
so that
\[
\partial_{\mathrm{Martin}}G(P) \cong \partial\mathbb D?
\]
\end{question}

\subsection{Variational viewpoints}


In the finite setting, \(L\) has a standard variational interpretation: 
the curvature one-form \(\sum_{i\in V} K_i\, du_i\), 
with \(u_i=\log r_i\), is closed and defines an energy whose gradient 
is the discrete curvature \(K\); in particular, 
\(L_i=K_i\) at the unit-radius configuration \(u\equiv 0\). 
For an infinite triangulation, however, the corresponding total energy 
need not be finite.

\begin{question}[Geometric characteristic number]
Is there a natural infinite-volume or unimodular energy functional 
whose first variation is given by the geometric characteristic \(L\)? 
More precisely, can the finite-dimensional variational structure 
of circle-pattern curvature be extended beyond finitely supported 
perturbations to the infinite unimodular setting?

\end{question}

\subsection{Higher-valence polyhedra}

The trivalent assumption is essential in this paper because the ADT
is a triangulation, so the angles determine the local triangular face
angles used in the definition of the geometric characteristic number. 
For more general hyperbolic polyhedra, the dual faces need not be triangles, and the intersection angles alone may no longer determine the corresponding face angles.

\begin{question}
Can the unimodular parabolic--hyperbolic dichotomy be extended from
trivalent hyperbolic polyhedra to more general infinite hyperbolic
polyhedral complexes?  If so, what curvature or geometric characteristic number should replace
\[
L_u(G,\Theta)=2\pi-\sum_{f\ni u}\theta_u^f
\]
when the dual complex is no longer a triangulation?
\end{question}

This problem has been partially addressed in the ideal case.  In
\cite{IIP2026}, the authors studied random infinite ideal angled
graphs (IAG) and ideal hyperbolic polyhedra (IHP).  
For such ideal polyhedra, the
appropriate geometric characteristic number is defined directly from the angles by
\[
T(u)=2\pi-\sum_{e\ni u}\Theta(e).
\]
We proved a corresponding unimodular parabolic--hyperbolic
dichotomy and developed the associated boundary theory.  Thus the
higher-valence problem is understood for infinite ideal hyperbolic
polyhedra.  However, for general  infinite hyperbolic polyhedra,
the correct replacement for the geometric characteristic number remains unclear.

\section*{Author information}

\begin{multicols}{2}

\AuthorBlock
  {Huabin Ge}
  {Renmin University of China}
  {hbge@ruc.edu.cn}

\AuthorBlock
  {Chuwen Wang}
  {Renmin University of China}
  {chuwenwang@ruc.edu.cn}

\vfill\null
\columnbreak

\AuthorBlock
  {Yangxiang Lu}
  {Renmin University of China}
  {2023000744@ruc.edu.cn}

\AuthorBlock
  {Tian Zhou}
  {Peking University}
  {2201110034@pku.edu.cn}

\end{multicols}

\end{document}